\documentclass[12pt]{amsart}
\usepackage{amssymb,latexsym, amscd, a4wide}
\usepackage[all]{xy}
\usepackage{graphicx}
\usepackage{mathrsfs}
\usepackage{amsmath}
\usepackage{pb-diagram}

 \newtheorem{theorem}{Theorem}[section]

 \newtheorem{cor}[theorem]{Corollary}

 \newtheorem{lem}[theorem]{Lemma}
\newtheorem{lemma-definition}[theorem]{Lemma-Definition}
\newtheorem{proposition-definition}[theorem]{Proposition-Definition}

\newtheorem{Convention}[theorem]{Convention}

 \newtheorem{proposition}[theorem]{Proposition} \theoremstyle{definition}
 \newtheorem{definition}[theorem]{Definition} \theoremstyle{definition}
\newtheorem{defi}[theorem]{Definition} \theoremstyle{definition}
 \theoremstyle{definition}
 
 \theoremstyle{remark}
 \newtheorem{rem}[theorem]{Remark}
  
\newtheorem{rem*}[theorem]{Remark}
 \numberwithin{equation}{section}
 \newcommand{\nc}{\newcommand}
\nc{\ssn}{\subsection{}} \nc{\sssn}{\subsubsection{}}

\numberwithin{equation}{section}
\newcommand{\Ind}{\operatorname{Ind}}

\newcommand{\BSx}{\ensuremath{{{BS}_{\ul{x}}}}}

\newcommand{\Kx}{{K}_{\ul{x}}}

\newcommand{\ul}{\ensuremath{\underline}}
\renewcommand{\k}{\mathbb{C}}
\newcommand{\Dd}{\mathbb{D}}
\newcommand{\df}{\overset{\operatorname{def}}{=}}

\renewcommand{\L}{\operatorname{L}}
\renewcommand{\P}{\operatorname{P}}

\newcommand{\J}{\mathcal{J}}

\newcommand{\Ac}{{\mathcal{A}}}

\newcommand{\Hom}{\operatorname{Hom}}

\newcommand{\End}{\operatorname{End}}
\newcommand{\Ext}{\operatorname{Ext}}

\newcommand{\Coker}{\operatorname{Coker}}

\newcommand{\Mod}{\operatorname{Mod}}
\newcommand{\Tr}{\operatorname{Tr}}
\newcommand{\DTr}{\mathbb{D}\operatorname{Tr}}

\newcommand{\Cone}{\operatorname{Cone}}

\newcommand{\tto}{\twoheadrightarrow}
\newcommand{\into}{\hookrightarrow}

\newcommand{\Proj}{\ensuremath{\operatorname{Proj}}}

\newcommand{\Inj}{\ensuremath{\operatorname{Inj}}}

\newcommand{\K}{\ensuremath{\operatorname{K}^-}}

\newcommand{\D}{\ensuremath{\operatorname{D}^-}}

\newcommand{\N}{\ensuremath{\mathbb{N}}}
\newcommand{\Z}{\ensuremath{\mathbb{Z}}}

\renewcommand{\mod}{\ensuremath{\operatorname{mod}}}

\newcommand{\C}{\ensuremath{\mathcal{C}}}
\newcommand{\B}{\ensuremath{\mathcal{B}}}

\newcommand{\Kb}{\ensuremath{\operatorname{K}^b}}

\newcommand{\bA}{\ensuremath{{\mathbf{A}}}}
\newcommand{\bB}{\ensuremath{{\mathbf{B}}}}
\newcommand{\bC}{\ensuremath{{\mathbf{C}}}}

\newcommand{\taul}{\ensuremath{\tau^{\leq 0}}}
\newcommand{\taug}{\ensuremath{\tau^{\geq 0}}}
\newcommand{\taulA}{\ensuremath{\tau^{\leq 0}_{\Ac}}}
\newcommand{\taugA}{\ensuremath{\tau^{\geq 0}_{\Ac}}}
\newcommand{\Tl}{\ensuremath{{\T}^{\leq 0}}}
\newcommand{\Tg}{\ensuremath{\T^{\geq 0}}}
\newcommand{\TlA}{\ensuremath{\T^{\leq 0}_{\Ac}}}
\newcommand{\TgA}{\ensuremath{\T^{\geq 0}_{\Ac}}}
\newcommand{\TlC}{\ensuremath{\T^{\leq 0}_{\C}}}
\newcommand{\TgC}{\ensuremath{\T^{\geq 0}_{\C}}}
\newcommand{\TlB}{\ensuremath{\T^{\leq 0}_{\B}}}
\newcommand{\TgB}{\ensuremath{\T^{\geq 0}_{\B}}}

\newcommand{\isoto}{\ensuremath{\overset{\sim}{\longrightarrow}}}

\newcommand{\TB}{\ensuremath{\T^b_{\B}}}
\newcommand{\TC}{\ensuremath{\T_{\C}}}
\newcommand{\TA}{\ensuremath{\T_{\Ac}}}

\newcommand{\Db}{\ensuremath{\operatorname{D}^b}}
\newcommand{\rad}{\ensuremath{\operatorname{rad}}}

\newcommand{\Ker}{\operatorname{Ker}}
\newcommand{\Jm}{\operatorname{Im}}

\newcommand{\ot}{\operatorname{\otimes}}

\newcommand{\g}{\ensuremath{\mathfrak{g}}}
\newcommand{\h}{\ensuremath{\mathfrak{h}}}
\renewcommand{\b}{\ensuremath{\mathfrak{b}}}

\newcommand{\T}{\ensuremath{\mathcal{T}}}
\newcommand{\Cab}{\ensuremath{\mathcal{C}}}
\newcommand{\KC}{\ensuremath{\operatorname{K}^{-}(\mathcal{C})}}

\newcommand{\Ss}{\ensuremath{{\mathbb{S}}}}
\newcommand{\SC}{\ensuremath{{\mathbb{S}_\C}}}
\newcommand{\SO}{\ensuremath{{\mathbb{S}_\BGG}}}

\newcommand{\SA}{\ensuremath{{\mathbb{S}_{\Ac}}}}
\newcommand{\V}{\ensuremath{\mathbb{V}}}

\newcommand{\BGG}{\ensuremath{\mathcal{O}}}

\newcommand{\A}{\ensuremath{{\mathcal{H}}}}

\newcommand{\AR}{\ensuremath{\mathcal{H}_{R}}}
\renewcommand{\AA}{\ensuremath{\mathcal{H}_{\Ac}}}
\newcommand{\AC}{\ensuremath{\mathcal{H}_{\C}}}
\newcommand{\AB}{\ensuremath{\mathcal{H}_{\B}}}

\newcommand{\m}{\ensuremath{\mathfrak{m}}}

\newcommand{\pa}{\ensuremath{\partial}}

\begin{document}
\author{Juan Camilo Arias Uribe and Erik Backelin}
\title[Higher Auslander-Reiten sequences]{Higher Auslander-Reiten sequences and $t$-structures.}
\address{Juan Camilo Arias Uribe, Departamento de Matem\'{a}ticas, Universidad de Los Andes,
Carrera 1 N. 18A - 10, Bogot\'a, Colombia}
\email{jc.arias147@uniandes.edu.co}
\address{Erik Backelin, Departamento de Matem\'{a}ticas, Universidad de Los Andes,
Carrera 1 N. 18A - 10, Bogot\'a, COLOMBIA}
\email{erbackel@uniandes.edu.co}

\keywords{}

\begin{abstract}
Let $R$ be an artin algebra and $\C$ an additive subcategory of
$\mod(R)$. We construct a $t$-structure on the homotopy category
$\KC$ and argue that its heart $\AC$ is a natural domain for
higher Auslander-Reiten (AR) theory. In the paper \cite{BJ} we
showed that $\operatorname{K}^{-}({\mod(R)})$ is a natural domain
for classical AR theory. Here we show that the abelian categories
$\A_{\mod(R)}$ and $\AC$ interact via various functors. If  $\C$
is functorially finite then $\AC$ is a quotient category of
$\A_{\mod(R)}$. We illustrate our theory with two examples:

When $\Cab$ is a maximal $n$-orthogonal subcategory Iyama
developed a higher AR theory,  see \cite{I}.  In this case we show
that the simple objects of $\AC$ correspond  to Iyama's higher AR
sequences and derive his higher AR duality from the existence of a
Serre functor on the derived category $\Db(\AC)$.

The category $\BGG$ of a complex semi-simple Lie algebra fits
into higher AR theory in the situation when $R$ is the coinvariant
algebra of the Weyl group.
\end{abstract}

\maketitle
\section{Introduction}
\subsection{{}}
In Auslander-Reiten (AR) theory one studies AR (or almost split)
short exact sequences in an abelian category $\Ac$. Throughout
this paper $R$ denotes an artin algebra. Typically $\Ac =
\mod(R)$, the category of finitely generated left $R$-modules, but
authors have also considered modules over a complete local
noetherian ring and coherent sheaves on a projective variety, see
e.g.  \cite{I2}, \cite{J} and references therein.

Let $\C$ be a full additive Karoubi closed subcategory  of $\Ac$.
By a higher dimensional AR theory in $\C$ we mean, loosely
speaking, a theory that resembles classical AR theory where the AR
sequences are replaced by certain complexes in $\C$. These
complexes should be ``almost split" or minimal in some sense. Such
a theory naturally takes place inside the triangulated category
$\TC \df \KC$ of bounded above complexes in $\C$ modulo homotopy.
We develop the theory under the additional assumption that $\C$ is
\emph{$\Ac$-approximating}. This means that for any object $X \in
\TA$ there is an object $X_\C \in \TC$ and a quasi-isomorphism
$X_\C \to X$ which induces an isomorphism $\Hom_{\TC}(\- , X_\C)
\isoto \Hom_{\TA}(\- , X)|_{\TC}$ of functors. The guiding example
is $\C = \Proj(\Ac)$ the full subcategory of projectives in $\Ac$
in the case when $\Ac$ has enough projectives.

\medskip

The theory we here develop is modelled on the approach to
classical AR theory in \cite{BJ}. We here briefly review this
theory. A more detailed review is given in Section \ref{Section
Preliminaries}. On $\TA$ there is a $t$-structure $(\TlA,\TgA)$ where $\TlA$ (resp. $\TgA$) is the strictly (i.e. closed under isomorphisms) full subcategory 
generated by $\{X \in \TA\mid  X^i = 0 \hbox{ for } i >0\}$ 
(resp.  $\{X \in \TA\mid X^i = 0 \hbox{ for } i < -2 \hbox{ and }  H^{-2}(X) = H^{-1}(X) = 0 \}$).
Its heart $\AA \df \TlA \cap \TgA$ is
the abelian category whose objects are isomorphic to three-term sequences $X =
[X^{-2} \to X^{-1} \to X^0]$ with vanishing cohomology in negative
degrees. In the case when $\Ac = \mod(R)$ then $X$ is simple in
$\AA$ if and only if $X$ is an AR sequence in the usual sense,
provided that $X^0$ is non-projective and indecomposable.
Moreover, AR duality and the existence of AR sequences is derived from
a Serre functor $\SA: \Proj(\AA) \isoto \Inj(\AA)$, where
$\Inj(\AA)$ denotes the full subcategory of injectives in $\AA$.
(This is influenced by ideas of Krause, \cite{K},\cite{K2}.)

\medskip

In this paper we will generalize this to a higher dimensional
theory on an $\Ac$-approximating subcategory $\C$ as follows:
There exists a $t$-structure $(\TlC,\TgC)$ on $\TC$  with $\TlC
\df  \{X \in \TC\mid  X^i = 0 \hbox{ for }  i >0\}$. (At this
level of generality
 $\TgC$ can only be described as the right orthogonal complement of
 $\TC^{<0}$.)
We study this $t$-structure, its heart $\AC$ and argue that this
provides a convenient framework for higher AR theory. In this
setting simple objects of the abelian category $\AC$ serve as
higher AR sequences and higher AR duality becomes a form of Serre
duality.

\subsection{{}}
In Section 3 we merely assume that $\C$ is any $\Ac$-approximating
subcategory where $\Ac$ is an arbitrary abelian category. We study
the heart $\AC$ and some natural functors between it and $\AA$. We
say that an abelian category has enough simples if each projective
has a simple quotient object.  We show  that if
$\AA$ has enough simples then so does $\AC$ (Proposition \ref{simple
objects prop} and Corollary \ref{simple objects cor}). We also show that
there is an equivalence
$$\P: \C \isoto \Proj(\AC), \, M \mapsto \P_M$$
where $\P_M $ is the complex with $M$ concentrated in degree $0$.
We show that $\AC$ has enough projectives and so it follows that
$\TC \cong \D(\AC)$ (Proposition \ref{Projectives in AC prop}). We
also relate our construction with the Yoneda embedding, which is a
standard tool in classical AR theory.

 \medskip

 In order to obtain more specific results we will from Section \ref{Duality Big Section} and onwards make the
 assumption that $\Ac =\mod(R)$.
 This assumption guarantees that all $\Hom$'s in the categories $\AA$ and $\AC$ are finitely generated
 modules over the center of $R$ and provides a duality functor $\Dd: \mod(R) \to \mod(R^{op})$.
 We show  that the Serre  functor $\SA$ induces a Serre functor $\SC: \Proj(\AC) \isoto \Inj(\AC)$, i.e. a functor together with an isomorphism
 $\Hom_{\AC}(\P_M, \- ) \cong \Dd \Hom_{\AC}(\- , \Ss_\C \P_M)$ (Proposition \ref{Duality prop}). We then show that the simple quotient $\L_M$ of $\P_M$ is isomorphic to the
 image of a certain ``minimal'' map $\P_M \to \Ss_\C \P_M$ (Corollary \ref{Duality cor}).

\medskip

The category $\C$ is called\emph{ functorially finite} if $\C$ is
$\Ac$-approximating and dually $\C^{op}$ is
$\Ac^{op}$-approximating. This notion occurs frequently in the
literature (see \cite{AS}, \cite{BR}, \cite{I},\cite{I2}).  For us its main importance
is that it implies that $\AC$ has enough injectives and that these
are all of the form $\SC \P_C$, for $C \in \C$. These facts are
proved in Proposition \ref{injectivesec prop}. In Theorem
\ref{injectivesec thm}  we use them to prove that $\pi_\C$ has a
right adjoint and is a quotient functor.

All this could be summarized as a slogan: ``higher AR theory takes
place in a quotient category of classical AR theory."

\medskip

We say that $\Ac$ has finite $\C$-dimension if all objects of
$\Ac$ admit $\C$-resolutions of uniformly bounded length. This is
equivalent to finiteness of the cohomological dimension of $\AC$
(Corollary \ref{Projectives in AC cor}). Under this hypothesis
$\TlC \cap \T^b_\C$ is a $t$-structure on $ \T^b_\C$ with the same
heart $\AC$. Moreover, the Serre functor induces
 an equivalence of triangulated
categories $\SC:  \T^b_\C \isoto \T^b_\C$ (Corollary
\ref{injectivesec cor}) which is a pleasant way to think of higher
AR duality.

 \medskip

 We give two examples of  higher AR theories. Both are functorially finite and of finite
 $\C$-dimension.
\subsection{{}}
Our first example is Iyama's generalized AR theory, \cite{I}, \cite{I2}.
We assume in Section \ref{Oyama's own section} that $\C$ is a
maximal $n$-orthogonal subcategory of $\mod(R)$. In this case an
object of $\AC$
 is an exact sequence $V = [C^{-n-2} \hookrightarrow \ldots \to C^{-1} \to C^0]$. This object is simple in $\AC$ if and only if it is an almost split sequence
 in the sense of Iyama. Our key observation is that the injectives  in $\AC$ take the
specific form:
 $$
\SC \P_X =  [X' \to I^{-n-1} \to \ldots \to I^0]
 $$
where the $I^j$'s are injective $R$-modules (Proposition
\ref{again my sweet injectives prop}). Now Serre duality leads to
a duality between $X$ and $X'$ which can be phrased as $X' \cong
\DTr \Omega^n X$, where $\Omega^n$ is the $n$'th syzygy in a
projective resolution of $X$ and $\DTr$ is Auslander and Reiten's
dual of the transpose. Moreover, if $V$ is simple and $C^0$ is
indecomposable and non-projective we retrieve Iyama's original
formula for AR duality: $C^{-n-2} = \DTr \Omega^n C^0$. This is
showed in Theorem \ref{really really AR thm}. The material in this
section is a shortened version of the master thesis \cite{A} of
the first author.

\subsection{{}}
In Section \ref{section frobenius algebras} we discuss the special
duality features that arise when $R$ is a Frobenius algebra. We
apply those results in Section \ref{section cat bgg} to the
category $\BGG$ of Bernstein-Gelfand-Gelfand of a complex
semi-simple Lie algebra $\g$. By theory of Soergel, see \cite{S},
the derived category $\Db(\BGG_0)$ of the principal block $\BGG_0
\subset \BGG$ is equivalent to a full subcategory of
$\T^b_{\mod(R)}$ where  $R$ is the coinvariant algebra of the Weyl
group of $\g$. The coinvariant algebra is Frobenius. In Theorem
\ref{Soergel's theory thm}  this allows us to interpret $\BGG$ as
a quotient category of $\A_{\mod(R)}$. This provides an
interesting relationship between $\BGG$ and AR theory in $\mod(R)$
that deserves to be investigated further. In particular, there
seem to be some interesting links between AR theory and the so
called dual Rouquier complex (see \cite{EW}) which is an
interesting object since the cohomology of its coinvariants
calculates higher extensions of Verma modules. See Sections
\ref{multisar} and \ref{Verma ext}.

\subsection{Acknowledgements} We like to thank Paul Bressler and
Kobi Kremnizer for useful conversations. We thank the referee for
valuable suggestions and for pointing out an error in an
earlier version.

\pagebreak

\section{Preliminaries}\label{Section Preliminaries}
In this section we fix notations and review the results of
\cite{BJ}.
\subsection{{Terminology}}
Throughout this paper we shall use the following terminology and
notations:

\medskip

\noindent Let $R$ be an artin algebra. Thus $R$ is a finite
algebra over its center $k$ which is a commutative artinian ring.
Let $\mod(R)$ be the category of all finitely generated left
$R$-modules.

Let $E$ be an injective hull of $k$ as a $k$-module. Then we have
the duality functor $M \mapsto \Dd M \df \Hom_k(M,E)$ that
interchanges finitely generated left and right $R$-modules; its
square is isomorphic to the identity. We also use the symbol $\Dd$
for the duality functor $\Hom_k(\ , E)$ on $\mod(k) =
\mod(k^{op})$. There is also the functor $M \mapsto M^* \df
\Hom_R(M,R)$ that again interchanges left and right $R$-modules.
We have $P \cong P^{**}$ whenever $P$ is a (finitely generated)
projective.

Let $P^{-1} \overset{\pa}{\longrightarrow} P^0 \to M \to 0$ be a
minimal projective presentation of some $M \in \mod(R)$. Then
define $\Tr M = \Coker \pa^* \in \mod(R^{op})$ and $\DTr M \in
\mod(R)$  (see \cite{ARS}).

\medskip

\noindent Let $\Ac$ denote a small abelian category, $\Proj(\Ac)$
and $\Inj(\Ac)$ the full subcategories of $\Ac$ consisting of
projectives and injectives, respectively.

\medskip

\noindent Let $\C$ denote a full additive subcategory of $\Ac$. We
assume furthermore that $\C$ is\emph{ Karoubi closed} (i.e. direct
summands in $\Ac$ of $\C$-objects belongs to $\C$) and
\emph{strictly full} (i.e. closed under isomorphisms). Let
$C^{-}(\C)$ be the category of bounded above complexes in $\C$.
Let $\TC \df \K(\C)$ be the triangulated category of bounded above
complexes in $\C$ modulo homotopy and $\T^b_\C \df \Kb(\C)$ its
subcategory of bounded complexes.

For an object $M \in \C$ and $n \in \mathbb{N}$ we let $M[n] \in
\TC$ denote the complex with $M$ concentrated in degree $-n$.

\medskip

\noindent Let $\T$ be a triangulated category and $\S$ a full
subcategory. $\S$ is \emph{closed under extensions} if for any
triangle $A \to B \to C \overset{+1}{\longrightarrow}$ in $\T$
with $A, C \in \mathcal{S}$ it follows that $B \in \mathcal{S}$.
$\mathcal{S}$ is closed under left (resp. right) shifts if
$\mathcal{S}[1] \subseteq \mathcal{S}$ (resp. $\mathcal{S}[-1]
\subseteq \mathcal{S}$). The right complement $\S^\perp$ (resp.
left complement ${}^\perp\S$) of $\S$ is the full subcategory of
$\T$ whose objects are $ \{X \in \T\mid  \Hom_\T(\S, X) = 0\}$
(resp. $ \{X \in \T\mid \Hom_\T(X,\S) = 0\}$).

\begin{defi}\label{Notations on triangulated categories defi} (See \cite{BBD}.) A
\emph{$t$-structure } on a triangulated category $\T$ is a pair of
full additive subcategories $(\Tl, \Tg)$ such that $\Tl$ is closed
under left shifts, $\Hom_\T(\Tl, \Tg[-1]) = 0$ and each $M \in \T$
fits into a triangle $M' \to M \to M''
\overset{+1}{\longrightarrow}$ with $M' \in \Tl$ and $M'' \in
\Tg[-1]$.
\end{defi}
The heart $\A \df \Tl \cap \Tg$ of the $t$-structure is an abelian
category. The inclusion functor $\Tl \to \T$ (resp. $\Tg \to \T$)
admits a right adjoint $\taul: \T \to \Tl$ (resp. left adjoint
$\taug: \T \to \Tg$). Put $\T^{\leq n} \df \Tl[-n]$, $\tau^{\leq n}
= [-n] \circ \taul \circ [n]: \T \to \T^{\leq n}$ and similarly
define $\T^{\geq n}, \tau^{\geq n}$, for $n \in \N$. It is known
that $\T^{\geq 1} = (\Tl)^\perp$, $\T^{\leq 1} = {}^\perp(\Tg)$
and $\Tl$ is closed under extensions.

Conversely,  assume that $\Tl \subseteq \T$
is a full subcategory closed under extensions and left shifts
and that the inclusion functor $\Tl \to \T$ admits a right
adjoint.  Then the pair $(\Tl, (\T^{\leq -1})^\perp)$  is a $t$-structure (see \cite{KV}).  \emph{We shall slightly
abuse notations and refer to such a subcategory $\Tl$ as a
$t$-structure.}
\subsection{{}}\label{roosterlabel-1} In this section we shall
discuss a $t$-structure $\Tl_\Ac$ on $\TA$ which is the strictly
full subcategory generated by
$$\{X \in \TA\mid   X^i=0 \hbox{ for } i>0\}.$$
This $t$-structure is completely parallel to the
$t$-structure $\TlA \cap \T^b_\Ac$ on the bounded homotopy
category $\T^b_\Ac$ that was introduced in \cite{BJ}, Section 2.2.
They have the same heart. For the purposes of this paper it turns
out that we need to work in the unbounded category $\TA$ however.

\begin{lem}\label{TlAc is t-structure lem}
$\Tl_\Ac$ is a $t$-structure on $\TA$.
\end{lem}
\begin{proof}
The argument from \cite{BJ}, Proposition 2.1, works also in $\TA$,
but for the sake of completeness we include a proof here. Clearly,
$\Tl_\Ac$ is closed under left shifts. To see that $\Tl_\Ac$  is
closed under extensions, let $A \to B \to C \overset{+1}{\longrightarrow}$ be
a triangle in $\TA$ with $A, C \in \Tl_\Ac$. This gives a morphism $C[-1] \to A$ in $\TA$. Chose a lift $f: C[-1] \to A$ of it to $\operatorname{C}^{-}(\Ac)$.
Thus $B \cong \Cone(f)$ in $\TA$. Since $\Cone(f)^i = C^i \oplus
A^i$ we conclude that $B \in \Tl_\Ac$. Define  a functor
$\taulA: \TA \to \TlA$ by
 $$\taul_\Ac X = [\ldots  \to X^{-2} \to X^{-1} \to \Ker (X^0 \to X^1) \to 0].$$
 It is straightforward to verify that $\taulA$ is right adjoint to the inclusion $\TlA \to \TA$.
\end{proof}
It is easy to see that $\Tg_\Ac$ is the strictly full subcategory
generated by 
$$\{X \in \TA\mid X^i = 0 \hbox{ for } i <-2 \hbox{ and }
H^{-2}(X) = H^{-1}(X) = 0\}$$ and that the truncation $\taugA$ is
given by
$$\taugA(X) = [0 \to \Ker d^{-1}_X \to X^{-1} \to X^0 \to \ldots ].$$

\medskip

\noindent Therefore the heart $\A_{\Ac}$ is the abelian category
whose objects are (homotopic to) complexes $[X^{-2} \to X^{-1} \to X^0]$ with no
cohomology except in degree $0$. Morphisms are morphisms of
complexes modulo homotopy.

For a morphism $f: X \to Y$ in $\A_{\Ac}$ let $\widetilde{f}: X
\to Y$ be a lift to a morphism in $\operatorname{C}^{-}(\Ac)$. We
have
\begin{equation}\label{KerA}
\Ker f \df \taulA(\Cone(\widetilde{f})[-1]) = [X^{-2} \to X^{-1}
\oplus Y^{-2} \to X^{0} \times_{Y^0} Y^{-1}],
\end{equation}
$$
\Coker f \df \taugA(\Cone(\widetilde{f})) = [X^0 \times_{Y^0}
Y^{-1} \to X^{0} \oplus Y^{-1} \to Y^0].
$$
Note that the localization functor $\TA \to \D(\Ac)$ maps $\TlA$
to the standard $t$-structure $D^{\leq 0} = \{X \in \D(\Ac)\mid\,
H^i(X) = 0 \hbox{ for }  i>0\}$ whose heart is $\Ac$. Thus the
functor $H^0: \AA \to \Ac$ is exact.

\medskip

\subsection{{}}\label{roosterlabel} In this section we recall the
most important properties of the abelian category $\AA$ from
\cite{BJ} and also prove an extension of one of them in Lemma \ref{roosterlabel lem}. Proofs for
the properties listed below can be found in \cite{BJ}, Corollary 2.5, Corollary 3.4, Proposition
4.6  and Section 4.
\begin{enumerate}
\item If $X^0$ is non-projective and indecomposable then $X^{-2} \hookrightarrow
X^{-1} \twoheadrightarrow X^0$ is an AR sequence in the
traditional sense iff $[X^{-2} \to X^{-1} \to X^0]$ is simple in
$\A_\Ac$.

\item The functor $\P: \Ac \to \Proj(\AA), \, M \mapsto \P_M \df M[0]$,
is an equivalence of categories. $\AA$ has enough
projectives and hence $\TA \cong \K(\Proj(\Ac))  \cong \D(\AA)$.

\item  There are exact sequences $0 \to \P_A \to \P_B \to \P_\C
\to [A \to B \to C] \to 0$; hence $\operatorname{gl.dim}(\AA) \leq
2$. (With equality unless $\Ac$ is semi-simple.)

\item If $\Ac$ has enough injectives then $\AA$ has enough
injectives and they are all of the form $[A \to I \to J]$ where
$I,J \in \Inj(\Ac)$.
\end{enumerate}

Assume furthermore that $\Ac = \mod(R)$. Then we have:

\begin{enumerate}
\item[(5)] $\Hom$'s in $\AA$ are finitely generated $k$-modules.
(However
 $\AA$ is neither a noetherian nor an artinian category unless $R$ has
finite representation type.)

\item[(6)] Let $M \in \Ac$ be indecomposable. Then $\P_M$ has a
unique simple quotient $\L^\Ac_M$ in $\AA$. If $M \notin \Proj(\Ac)$ then
$\L^\Ac_M$ equals an AR sequence $[\DTr M \to N \to M]$. If $M \in
\Proj(\Ac)$ then $\L^\Ac_M = [0 \to \rad M \to M]$.

\item[(7)] Let $M \in \Ac$. Then there is a unique object $\SA
\P_M \in \Inj(\AA)$ such that
$$\Dd \Hom_{\AA}(\P_M, \- ) \cong \Hom_{\AA}(\- , \SA \P_M).$$ This defines a functor $\SA: \Proj(\AA) \to \Inj(\AA)$ which is an equivalence.

\item[(8)] If $M \in \Ac$ contain no projective direct summand we
have $\SA \P_M = [\DTr M \to I \to J]$ where $0 \to \DTr M \to I
\to J$ is a minimal injective corepresentation. If $M \in
\Proj(\Ac)$ we have $\SA \P_M = \P_{\Dd(M^*)}$. (Note that in the
latter case $M^*$ is a projective right $R$-module so that
$\Dd(M^*) \in \Inj(\Ac)$ and hence $\P_{\Dd(M^*)} \in \Inj(\AA)$.)
\end{enumerate}

Later on we shall need an extension of the isomorphism in item (7)
from $\AA$ to $\TA$. The following result can be deduced from the
extension of $\SA$ to $\T^b_\Ac$ established in \cite{BJ},
Proposition 4.6.  We opted however to give an argument
that works directly in $\TA$.
\begin{lem}\label{roosterlabel lem} There is a functorial isomorphism
$\Dd \Hom_{\TA}(\P_M, V) \cong \Hom_{\TA}(V, \SA \P_M)$ for $M \in
\Ac, V \in \TA$. It extends the isomorphism in  item (7).
\end{lem}
\begin{proof} We first show that
\begin{equation}\label{serreextended1}
\P_M \in {}^\perp(\TA^{\geq 1}) \cap {}^\perp(\TA^{\leq -1}) \
\hbox{ and }  \ \SA \P_M \in (\TA^{\geq 1})^\perp \cap (\TA^{\leq
-1})^\perp.
\end{equation}
We have  $\P_M \in \AA \subset \TA^{\leq 0} = {}^\perp(\TA^{\geq
1})$ and $\P_M \in {}^\perp(\TA^{\leq -1})$ holds trivially since
$\P_M$ has no components in negative degrees. Also, we have $\SA
\P_M \in \AA \subset \TA^{\geq 0} = (\TA^{\leq -1})^\perp$.

The only assertion that needs a proof is $\SA \P_M \in (\TA^{\geq
1})^\perp$. Since $\SA \P_M$ is injective we have  by item (4)
above that $\SA \P_M = [A \overset{d^{-2}}{\longrightarrow} I
\overset{d^{-1}}{\longrightarrow} J]$ with vanishing cohomology in
negative degrees, $A \in \Ac$ and $I, J \in \Inj(\Ac)$. Let $X \in
\TA^{\geq 1}$ and $f \in \Hom_{\TA}(X, \SA \P_M)$. We must prove
that $f=0$. We may assume
$$X= \ldots \to 0 \to X^{-1} \overset{d^{-1}_X}{\longrightarrow} X^0
\overset{d^{0}_X}{\longrightarrow} X^1 \to \ldots,$$ with
vanishing cohomology in degrees $\leq 0$. Hence $f$ is given by a
pair $(f^{-1}: X^{-1} \to I, f^{0}: X^0 \to J)$ such that
$d^{-1}f^{-1} = f^0d^{-1}_X$. Since $d^{-1}_X$ is injective and
$I$ is injective there is a map $\alpha^0: X^0 \to I$ such that
$\alpha^0 d^{-1}_X = f^{-1}$. Therefore $(f^0 -
d^{-1}\alpha^0)\circ d^{-1}_X = 0$ and hence the injectivity of
$J$ shows that there is a map $\alpha^1: X^1 \to J$ such that $f^0
- d^{-1}\alpha^0 = \alpha^1d^0_X$. Hence $\alpha = (\alpha^{0},
\alpha^1)$ defines a homotopy $f \sim 0$.

Consider now the triangles $\tau^{\leq 0}_\Ac V \to V \to
\tau^{\geq 1}_\Ac V \overset{+1}{\longrightarrow}$ and $\tau^{\leq -1}_\Ac V
\to \tau^{\leq 0}_\Ac V \to \tau^{\leq 0}_\Ac \tau^{\geq 0}_\Ac V
\overset{+1}{\longrightarrow}$. By \eqref{serreextended1} there
are natural isomorphisms $$\Hom_{\TA}(\P_M,V)
\overset{\sim}{\leftarrow} \Hom_{\TA}(\P_M,\tau^{\leq 0}_{\Ac} V)
\overset{\sim}{\rightarrow} \Hom_{\TA}(\P_M, \tau^{\leq
0}_{\Ac}\tau^{\geq 0}_{\Ac} V)$$ and
 $$\Hom_{\TA}( V, \SA \P_M)
\overset{\sim}{\leftarrow} \Hom_{\TA}(\tau^{\leq 0}_{\Ac} V, \SA
\P_M) \overset{\sim}{\rightarrow}  \Hom_{\TA}(\tau^{\leq 0}_{\Ac}
\tau^{\geq 0}_{\Ac}V, \SA \P_M).$$ On the other hand, since
$\tau^{\leq 0}_\Ac \tau^{\geq 0}_\Ac V \in \AA$ we get from item
(7) above the isomorphism $$\Dd \Hom_{\TA}(\P_M, \tau^{\leq 0}_\Ac
\tau^{\geq 0}_\Ac V) \cong \Hom_{\TA}(\tau^{\leq 0}_\Ac \tau^{\geq
0}_\Ac V, \SA \P_M).$$ The lemma follows.
\end{proof}

\section{Higher AR theory and a $t$-structure on $\TC$}\label{big t Section}
We here begin to study higher AR theory on an additive category
$\C$ by means of a $t$-structure on the triangulated category
$\TC$. Throughout this paper $\C$ will denote a full subcategory
of the abelian category $\Ac$. Thus $\TC$ is a full triangulated
subcategory of $\TA$.

\subsection{{}} \label{ACdefsec}

\begin{defi}\label{big t Section defi}
Let $M \in \Ac$ and $X \in \TA$.
\end{defi}
 \begin{enumerate}
\item A $\C$-\emph{cover} (often called a surjective $\C$-precover
in the literature) of $M$ is a surjective $\Ac$-morphism $C
\twoheadrightarrow M$, where $C \in \C$, that induces a surjection
$\Hom_{\C}(\- , C) \twoheadrightarrow \Hom_{\Ac}(\- , M)|_{\C}$.

\item A $\C$-\emph{resolution} of $M$ is an exact sequence $\ldots
\to C^{-n} \to \ldots \to C^0 \tto M$ such that $\ldots \to
\Hom_\C( \- , C^{-n}) \to \ldots \to \Hom_\C( \- , C^0) \tto
\Hom_{\Ac}(\- , M)|_\C$ is exact.

\item A $\C$-\emph{approximation} of $X$  is a quasi-isomorphism
 $X_\C \to X$ which induces an isomorphism $\Hom_{\TC}(\- , X_\C) \isoto \Hom_{\TA}(\- , X)|_{\TC}$, where $X_\C \in \TC$.

\item Dually, we define a contravariant  $\C$-resolution (resp. a
contravariant $\C$-approximation)
 to be a $\C^{op}$-resolution (resp. a $\C^{op}$-approximation in $\T_{\C^{op}}$).
\end{enumerate}
\begin{rem}\label{projrem}
Assume that $\Ac$ has enough projectives and that $\Proj(\Ac)
\subseteq \C$. If $X_\C \to X$ is a morphism in $\TA$, with $X_\C
\in \TC$, such that $\Hom_{\TC}(\-, X_\C) \to \Hom_{\TA}(\- ,
X)|_\C$ is an isomorphism, then  $X_\C \to X$ is a
quasi-isomorphism. This follows from the isomorphisms
$\Hom_{\TC}(P[n], X_\C) \isoto \Hom_{\TA}(P[n], X)|_\C$ for all $P
\in \Proj(\Ac)$, $n \in \mathbb{Z}$.
\end{rem}

\begin{rem}\label{defagainrem} Whenever a $\C$-approximation $X_\C \to X$ exists it is unique up to
canonical isomorphism. This follows from the Yoneda lemma because
if $X'_\C \to X$ is another $\C$-approximation then we are given
isomorphisms $\Hom_{\TC}(\- , X_\C) \cong \Hom_{\TA}(\- ,
X)|_{\TC} \cong \Hom_{\TC}(\- , X'_\C)$ and hence an isomorphism
$X_\C \cong X'_\C$ over $X$. Note also that if $C
\twoheadrightarrow M$ is a $\C$-cover then any map $C' \to M$
factors through it. Therefore $\C$-resolutions behave pretty much
as projective resolutions.
\end{rem}
\begin{lem}\label{big t Section lem}
The following conditions are equivalent:
\begin{enumerate}
\item Each $M \in \Ac$ has a $\C$-cover. \item Each $M \in \Ac$
has a $\C$-resolution. \item Each $X \in \TA$ has a
$\C$-approximation.
\end{enumerate}
\end{lem}
\begin{proof}\noindent (1) $\implies$ (2). Let $C^0 \to M$ and $C^{-1} \to \Ker(C^0 \to M)$ be $\C$-covers
and define inductively $\C$-covers $C^{-n-2}\to \Ker(C^{-n-1} \to C^{-n})$. Then $C^\bullet \to M$ is a
$\C$-resolution because $\Hom_\Ac(\- , \-)$ is left exact.

\medskip
\noindent (2) $\implies$ (3). A $\C$-approximation $X_\C \to X$
can be inductively constructed by Godement's method: Let $N$ be he
largest index such that $X^N \neq 0$. We define $X^n_\C = 0$ for
$n>N$ and let $X^N_\C \in \C$ be a $\C$-cover of $X^N$. Let $n
\leq N$ and assume that $X^m_\C$ and morphisms $d^m: X^m_\C \to
X^{m+1}_\C$ have been constructed for all $m \geq n$. Let $F^{n-1}
= \Ker d^n \times_{X^n} X^{n-1}$ and let $X^{n-1}_\C \to F^{n-1}$
be a $\C$-cover. The composition $X^{n-1}_\C \to F^n \to \Ker d^n
\hookrightarrow X^n_\C$ defines $d^{n-1}$. It is straightforward
to construct the morphism $X_\C \to X$ and to verify that it is a
$\C$-approximation.

\medskip
\noindent  (3) $\implies$ (1).  Let $M \in \Ac$ and pick a
$\C$-approximation $M_\C \to M[0]$. We claim that the induced
map $M^0_\C \to M$ is a $\C$-cover. To see this note that the
isomorphisms $H^0(M_\C) \to M$ and $H^0(\Hom_{\TA}( \-
,M_\C)) \isoto \Hom_\Ac( \- ,M)$ implies that the natural
maps $\Ker(M^0_\C \to M^1_\C) \to M$ and $\Ker (\Hom_\Ac( \- ,M^0_\C) \to
\Hom_\Ac( \- ,M^1_\C)) \longrightarrow \Hom_\Ac( \- , M)$ are
surjective. Thus, $M^0_\C \to M$ and $\Hom_\Ac( \- ,M^0_\C) \to
\Hom_\Ac( \- , M)$ are surjective as well.
\end{proof}
Assume that every object of $\Ac$ admits a $\C$-resolution. Then
the $\C$-\emph{dimension} of $M \in \Ac$ is defined to be the
smallest number $n \in \mathbb{N} \cup \{\infty\}$ such that there
is a $\C$-resolution of the form $0 \to C^{-n} \to \ldots \to C^0
\to M \to 0$. The $\C$-dimension of $\Ac$ is the supremum of the
$\C$-dimensions of the objects of $\Ac$.

\begin{cor}\label{big t Section cor} Let $V \in \TA$ and suppose that $V^i = 0$ for all $i>k$. Then $V$ has a $\C$-approximation $V_\C$ such that $V^i_\C = 0$ for $i>k$.
If furthermore $\Ac$ has $\C$-dimension $n$ and $V^i = 0$ for all
$i$ outside some interval $[m,k]$ then $V^i_\C = 0$ for $i \notin
[m-n,k]$.
\end{cor}
\begin{proof} This is clear from the explicit construction of  a $\C$-approximation in the proof of $(2) \implies (3)$ in the above lemma.
\end{proof}

\begin{defi}\label{big t Section defi2} $\C$ is called $\Ac$-\emph{approximating} if each object of $\TA$ admits a $\C$-approximation.
Dually, $\C$ is called \emph{contravariantly} $\Ac$-approximating
if $\C^{op}$ is $\Ac^{op}$-approximating. If $\C$ is both
$\Ac$-approximating and contravariantly $\Ac$-approximating we
call it \emph{functorially finite}.

In the case when $\Ac = \mod(R)$ we write $\Dd \Ac = \mod(R^{op})$
and $\Dd \C = \{\Dd M\mid M \in \C\}$. Note that if $\Ac =
\mod(R)$ then $\C$ is contravariantly $\Ac$-approximating $\iff$
$\Dd \C$ is $\Dd \Ac$-approximating. We say  $R$-approximating
instead of $\Ac$-approximating when $\Ac = \mod(R)$.
\end{defi}
\medskip

\noindent {\bf{Guiding example.}} Let $\C = \Proj(\Ac)$. Then a
$\C$-resolution is a projective resolution and $\C$ is
$\Ac$-approximating iff $\Ac$ has enough projectives.

\begin{lem}\label{ACdefsec lem} Assume that $\C$ is $\Ac$-approximating. Then the inclusion functor
$\TC \hookrightarrow \TA$ has a right adjoint $\pi_\C: \TA \to
\TC$ which is a triangulated functor. The adjunction morphism $(\TC \hookrightarrow \TA) \circ
\pi_\C \to Id_{\TA}$ defines a $\C$-approximation $\pi_\C X \to X$
for each $X \in \TA$. Moreover, we can (and will) define $\pi_\C$
such that $\pi_\C |_{\TC} = Id_{\TC}$.
\end{lem}
\begin{proof}
Choose for each $X \in \TA$ a $\C$-approximation $\pi_\C X  \to X$
such that for $X \in \TC$ we have $\pi_\C X = X$.  Observe that
for each morphism $f:X \to Y$ in $\TA$ there is a unique morphism
$\pi_\C f: \pi_\C X \to \pi_\C Y$ that makes $\pi_\C$ a functor
and $(\TC \hookrightarrow \TA) \circ \pi_\C \to Id_{\TA}$ a
natural transformation. By construction  $\pi_\C$  is right
adjoint to $(\TC \hookrightarrow \TA)$ and the previous natural
transformation is the adjunction morphism.
\end{proof}
\begin{Convention}\label{ACdefsec conv}
Whenever $\C$ is $\Ac$-approximating we shall tacitly assume that
choices of $\C$-approximations have been made defining a functor
$\pi_\C: \TA \to \TC$ which is right adjoint to the inclusion
$\TC \hookrightarrow \TA$, as in Lemma \ref{ACdefsec lem}. We write $X_\C = \pi_\C X$ for $X
\in \TA$.
\end{Convention}
Note that $\pi_\C \TlA \subseteq \TlA$ by corollary \ref{big t Section cor}.
\begin{defi}\label{ACdefsec defi}
Put $\Tl_\C \df  \TC \cap \Tl_\Ac$. Thus an object of  $\TlC$ is isomorphic to a complex $X \in \TC$ such that $X^i=0$ for
 $i>0$.
\end{defi}
\begin{proposition}\label{ACdefsec prop} Assume that $\C$ is $\Ac$-approximating.

\textbf{i)}  $\TlC$ is a $t$-structure with heart $\AC$ and
truncation functor $\taul_\C \df \pi_\C \circ \taulA|_{\TC}$.

\textbf{ii)}  Assume that  $\Ac$ has finite $\C$-dimension. Then
$\T^b_\C \cap \TlC$ is a $t$-structure on $\T^b_\C$ with heart
equivalent to $\AC$.
\end{proposition}
\begin{proof} \textbf{i)} Clearly $\TlC$ is closed under left shifts. It is closed under extension by the same argument as in Lemma \ref{TlAc is t-structure lem}.
Finally the inclusion $\TlC \to \TC$ has the right adjoint $\pi_\C
\circ \taulA|_{\TC}$.

\medskip

\noindent \textbf{ii)} The hypothesis implies that
$\pi_\C(\T^b_\Ac) = \T^b_\C$. It follows as above that the
inclusion  $\T^b_\C \cap \TlC \hookrightarrow \T^b_\C$ has a right
adjoint and the rest follows.
\end{proof}

\subsection{{}}\label{partial converse sds} There is a partial strengthening of and a converse to Proposition
\ref{ACdefsec prop}.\label{partial converse}
\begin{proposition}\label{partial converse sds prop} \textbf{i)} Assume that the kernel in $\Ac$ of each $\C$-morphism
admits a $\C$-cover. Then $\TlC$ is a $t$-structure.

\textbf{ii)} Assume that $\TlC$ is a $t$-structure.  Let $K$ be the kernel in $\Ac$ of a $\C$-morphism $f: M \to N$.
Then there is a $K_\C \in \TC$ and a $\TA$-morphism
$K_\C \to K[0]$ which induces an isomorphism $\Hom_{\TC}(\- , K_\C)
\isoto \Hom_{\TA}(\- , K[0])|_{\TC}$. Therefore if $K_\C \to K[0]$ in
addition happens to be a quasi-isomorphism then it is a
$\C$-approximation.

\textbf{iii)}   Assume that $\TlC$ is a $t$-structure, $\Ac$ has
enough projectives, $\Proj(\Ac) \subseteq \C$ and any object of
$\C$ is isomorphic to the kernel of some $\Ac$-morphism. Then $\C$
is $\Ac$-approximating.
\end{proposition}
\begin{proof} \textbf{i)}  The argument of   $(2) \implies (3)$ of Lemma
\ref{big t Section lem} applies in this situation and shows that $ \taulA X$ admits a $\C$-approximation $(\taulA X)_\C \to \taulA X$, for $X \in \TC$.
By the same argument as in the proof of Lemma \ref{ACdefsec lem} this allows us to define a functor $\taul_\C: \TC \to \TlC$,
satisfying $\taul_\C X \cong (\taulA X)_\C$, which is right adjoint to the inclusion $\TlC \to
\TC$.

\medskip

\noindent \textbf{ii)} Let $X$ be the complex $M
\overset{f}{\longrightarrow} N$ with $M$ in degree $0$ so that
$K[0] = \taulA X$. Put $K_\C \df \taul_\C X$. Then $\taulA K_\C =
K_\C$  and therefore we get by applying $\taulA$ to the canonical
morphism $K_\C \to X$ a morphism $K_\C \to  K[0]$. We must show
that the latter defines isomorphisms
\begin{equation}\label{star condition one}
 \Hom_{\TC}(V, K_\C) \isoto \Hom_{\TA}(V, K[0]), \hbox{
for all } V \in \TC.
\end{equation}
It is
enough to show that \eqref{star condition one} holds when $V =
C[i]$, for $C \in \C$ and $i \in \Z$. Assume first that $i \geq
0$. Then $C[i] \in \TlC$ and we therefore obtain
$$\Hom_{\TC}(C[i], K_\C) = \Hom_{\TC}(C[i], \taul_\C X) \cong \Hom_{\TC}(C[i], X) =
$$
$$
 \Hom_{\TA}(C[i], X) \cong   \Hom_{\TA}(C[i], \taulA X) = \Hom_{\TA}(C[i],K[0]).
$$
On the other hand if $i <0$ we have $\Hom_{\TC}(C[i], K_\C) =
\Hom_{\TA}(C[i], K[0]) = 0$ since $K_\C$ is (isomorphic to) a complex with no
non-zero components in strictly positive degrees.

\medskip

\noindent \textbf{iii)} This follows from \textbf{ii)} and Remark
\ref{projrem}.
\end{proof}
Although it isn't strictly necessary in order to obtain a $t$-structure we shall always work with the
notion of an $\Ac$-approximating subcategory rather than with the
weaker assumptions of \textbf{i)} in the proposition.

\subsection{{}}\label{Section AAo} For the rest of Section \ref{big t Section} we assume that $\C$ is $\Ac$-approximating.
We shall describe the $t$-structure $\TlC$ a bit closer.
\begin{lem}\label{Section AAo lem} \textbf{i)} We have $\taug_\C  \cong \pi_\C \circ \taug_\Ac |_{\TC}$.
\textbf{ii)}  $\TlC = \pi_\C(\TlA)$,  $\TgC = \pi_\C(\TgA)$  and consequently
$\pi_\C(\AA) = \AC$.
\end{lem}
\begin{proof} \textbf{i)} Let $X \in \TC$.
 Applying $\pi_\C$ to the triangle $\tau^{<0}_\Ac X \to X \to
\taug_\Ac X \overset{+1}{\longrightarrow}$ we get the triangle
$$
 \tau^{<0}_\C X \to X \to \pi_\C \taug_\Ac X
\overset{+1}{\longrightarrow}.
$$
Now the triangle $\tau^{<0}_\C X \to X \to \taug_\C X
\overset{+1}{\longrightarrow}$ shows that  $\pi_\C \taugA X \cong \taug_\C X$.

\medskip

\noindent  \textbf{ii)} From the inclusions $\TlC = \pi_\C(\TlC) \subseteq \pi_\C \TlA \subseteq \TlA \cap \TC \df \TlC$ we conclude that $\TlC =\pi_\C \TlA$.
We show that $ \TgC = \pi_\C(\TgA)$.  For $X \in \TgC$ we have $X \cong \taug_\C X \cong \pi_\C \taug_\Ac X \in \pi_\C \TgA$, by  \textbf{i)}.
On the other hand we have
$$\Hom_{\TC}(\T^{\leq -1}_{\C} ,  \pi_\C ( \TgA)) = \Hom_{\TA}(\T^{\leq -1}_{\C} ,   \TgA )  = 0,$$
since $\pi_\C$ is right adjoint to the inclusion $\TC \hookrightarrow \TA$ and $\T^{\leq -1}_{\C} \subseteq \T^{\leq -1}_{\Ac} = {}^\perp \TgA $. Thus $\pi_\C ( \TgA) \subseteq (\T^{\leq -1}_{\C} )^{\perp} = \TgC$ (where the right orthogonal complement is taken in $\TC$).
\end{proof}

\subsection{{}}\label{Section AA}
Let $f:X \to Y$ be a morphism in $\AC$ and let $\widetilde{f}: X
\to Y$ be a representative of it in $\operatorname{C}^{-}(\C)$.
By the definition of (co)kernels in the heart of a $t$-structure
$$\Ker f  = \taul_\C(\Cone(\widetilde{f})[-1]) \hbox{ and } \Coker f = \taug_\C(\Cone(\widetilde{f})).$$

\begin{Convention}\label{Section AA con}  From now on when we write ``let $X \in \TlC$" or ``let $X \in \AC$" we shall tacitly assume that $X^i = 0$ for all $i>0$ and not merely that $X$ is homotopic to such a complex. The significance of this assumption is the assertion of Lemma \ref{Section AA lem} i) below.
\end{Convention}
\begin{lem}\label{Section AA lem} Let $X,Y \in \AC$.
\textbf{i)} $X \cong 0 \iff X \in \T^{\leq -1}_\C \iff X^{-1} \to X^0$ is a split epimorphism.

\textbf{ii)} $f$ is surjective if and only if $X^0 \oplus Y^{-1}
\to Y^0$ is a split epimorphism.
\end{lem}
\begin{proof} \textbf{i)}  The first equivalence holds since $\T^{<0}_\C \cap \AC = 0$.
The last equivalence holds by the definition of $\T^{\leq -1}_\C$.

\medskip

\noindent \textbf{ii)} Assume that $X^0 \oplus Y^{-1} \to Y^0$ is
a split epimorphism. Then ${\tau}^{\geq 0}_{\Ac}
\Cone(\widetilde{f}) = 0$. Thus $\Coker f = \pi_\C({\tau}^{\geq 0}_{\Ac}
\Cone(\widetilde{f})) = 0$. Conversely, assume that $\Coker f =
0$.  Then we have
$$
\Hom_{\TA}(Y^0[0], {\tau}^{\geq 0}_{\Ac} \Cone(\widetilde{f})) =
\Hom_{\TC}(Y^0[0], \Coker f) = 0.
$$
Since $\taugA \Cone(\widetilde{f}) = [X^0 \times_{Y^0} Y^{-1}  \to X^0 \oplus Y^{-1} \to Y^0]$ this implies that $X^0 \oplus Y^{-1} \to Y^0$ is a split epimorphism.
\end{proof}

\begin{proposition}\label{Section AA prop} \textbf{i)} The inclusion $\TC \hookrightarrow \TA$ is right $t$-exact; hence  $\taug_{\Ac}(\AC) \subseteq \A_{\Ac}$ and $\taug_{\Ac}: \AC \to \AA$ is right exact.

\textbf{ii)}   $\taug_{\Ac} : \AC \to \AA$ is left adjoint to
$\pi_\C: \AA \to \AC$ and the adjunction morphism $Id_{\AC} \to
\pi_\C \circ  \taug_{\Ac} |_{\AC}$ is an isomorphism.

\textbf{iii)} $\pi_\C: \AA \to \AC$ is exact.

\textbf{iv)} Let $\bA =\Ker( \pi_\C|_{\AA})$. Let $\mu: \taugA \circ \pi_\C|_{\AA} \to Id_{\AA}$
be the adjunction morphism and let $\bB$ be the set of objects isomorphic to $\Coker(\mu_V)$, for $V \in
\AA$ (where the cokernel is calculated in $\AA$). Let $\bC$ be the set of objects isomorphic to $[X   \to C \times Y \to Z] \in \AA$, where $C \to Z$ is a  $\C$-cover.
Then $\bA = \bB = \bC$.
\end{proposition}
\begin{proof} \textbf{i)} holds by definition since $\Tl_\C \subseteq \Tl_\Ac$.

\medskip

\noindent \textbf{ii)} For $X \in \AC, Y \in \AA$ we have
$$
\Hom_{\AC}(X, \pi_\C Y) \cong \Hom_{\TA}(X, Y) \cong
\Hom_{\TA}(\taugA X, Y) = \Hom_{\AA}(\taugA X, Y)
$$
where the second isomorphism holds since $\Hom_{\TA}(\tau^{<0}_\Ac
X, Y) = 0$ by the assumption on $Y$. This proves the adjointness
of the functors. Similarly, for all $Z \in \AC$ we have
$$\Hom_{\TA}(Z,X) \cong \Hom_{\TA}(Z, \taug_\Ac X) \cong
\Hom_{\TC}(Z, \pi_\C \taug_\Ac X)$$ where the first isomorphism
follows from the fact that $\Hom_{\TA}(Z, \tau^{< 0}_{\Ac} X) =
\Hom_{\TC}(Z, \tau^{< 0}_{\C} X)$ $= \Hom_{\TC}(Z, 0) = 0$. Hence
$X \isoto \pi_\C \taug_\Ac X$.

\medskip

\noindent \textbf{iii)} The functor $\pi_\C: \AA \to \AC$ is left
exact since it has a left adjoint. We prove it is right exact. Let
$X \tto Y$ be a surjection in $\AA$. Then $X^0 \oplus Y^{-1} \to
Y^0$ is a split epimorphism (by Lemma \ref{Section AA lem} ii) applied to the special case $\C$ equals $\Ac$).  Using Godement's explicit
construction of $\C$-approximations as in the proof of $(1) \implies (2)$ of Lemma
\ref{ACdefsec lem} we see that $X^0_\C \oplus Y^{-1}_\C \to
Y^0_\C$ is a split epimorphism as well. Hence
$X_\C \to Y_\C$ is surjective, again by Lemma \ref{Section AA lem} ii).

\medskip

\noindent \textbf{iv)} We show that $\bA \subseteq \bB$. Let $V
\in \bA$. We get the exact sequence
$$
\taugA\pi_\C V \overset{\mu_V}{\longrightarrow} V \to \Coker \mu_V \to 0
$$
Since $\pi_\C V = 0$ we have $V \cong \Coker \mu_V \in \bB$.

\medskip

\noindent We prove that $\bB \subseteq \bC$. Let $V = [U \to Y \to
Z] \in \AA$ and consider $\Coker \mu_V \in \bB$. By Godement's
construction we can chose a $\C$-approximation  $V_\C \to V$ where
$V_\C = [ \ldots \to V^{-1}_\C \overset{d^{-1}}{\longrightarrow}
V^0_\C]$ and $V^0_\C \to Z$ is a $\C$-cover.  Then $\taugA V_\C =
[\Ker d^{-1} \to V^{-1}_\C \to V^0_\C]$. We thus get $\Coker \mu_V
\cong  [X  \to C \times Y \to Z]$, where $X =V^0 \times_Z Y$ and
$C = V^0_\C$. Thus $\Coker \mu_V \in \bC$.

\medskip

\noindent We prove that $\bC \subseteq \bA$.  Let
$V=[X   \to C \times Y \to Z] \in \bC$ where $C \to Z$ is a $\C$-cover.
Then we can pick a $\C$-approximation $V_\C$ is of the form $[\ldots \to V^{-1}_\C
\oplus C \to C]$ where $V^{-1}_\C \oplus C \to C$ is a split
epimorphism; thus $V_\C = 0$ by Lemma \ref{Section AA lem} i). Hence, $V \in \bA$.
\end{proof}
\begin{rem}\label{Section AA rem} Let $\C = \Proj(\Ac)$ and assume that $\Ac$ has enough projectives. Then $\AC \cong \Ac$ and $\pi_\C = H^0: \AA \to \Ac$. This functor $H^0$
is exact since the canonical map $\TA \to \D(\Ac)$ is $t$-exact
with respect to the standard $t$-structure on the derived
category.
\end{rem}

\subsection{Projectives  in $\AC$}\label{Projectives in AC}
For $M \in \C$ recall that $\P_M \in \A_\C$ is the complex $M$
concentrated in degree $0$. This gives a fully faithful functor
$$\P: \C \to \AC, \-  M \mapsto \P_M.$$
We have
\begin{proposition}\label{Projectives in AC prop} $\P$ defines an equivalence of categories $\P: \C \isoto  \Proj(\AC)$.  $\AC$ has enough projectives.
Hence, the natural morphism $\TC \to \D(\AC)$ is an equivalence of
categories.
\end{proposition}
\begin{proof}  We show $\P_M$ is projective. Let $f:X \to Y$ be a surjection in $\A_\C$ and $g:\P_M \to Y$ a morphism
(thus $g$ is given by $g^0: M \to Y^0$). We must show that $g$
factors through $f$. By Lemma \ref{Section AA lem} we know that
$X^0 \oplus Y^{-1} \to Y^0$ has a splitting $s: Y^0 \to X^0 \oplus
Y^{-1}$. Let $\pi:  X^0  \oplus  Y^{-1} \to X^0$ be the projection
and let $g': \P_M \to X$ be given by $g'^0 = \pi s g^0$. Then
$fg'$ is homotopic to $g$.

Now, since any $X$ is a quotient of $\P_{X^0}$ it follows that a
projective $X$ is a direct summand in  $\P_{X_0}$. Hence $X \cong
\P_M$ for some direct summand $M$ of $X_0$.
\end{proof}
It follows from the proposition that any $V = [\ldots \to C^{-1}
\to C^0] \in \AC$ admits a projective resolution
\begin{equation}\label{gl.dim0}
\ldots \to \P_{C^{-1}} \to \P_{C^0} \to V \to 0.
\end{equation}
We obtain
\begin{cor}\label{Projectives in AC cor}
$\Ac$ has finite $\C$-dimension $n$ $\implies$ $\AC$ has finite
cohomological dimension $\leq n+2$. Conversely, the $\C$-dimension
of $\Ac$ is bounded by the cohomological dimension of $\AC$.
\end{cor}
\begin{proof} Let $V
\in \AC$. Thus $V \cong U_\C$ for some $U \in \AA$.  Since $U^i =
0$ for $i \notin [-2,0]$ we conclude that $V$ is isomorphic to an
object of the form $[C^{-n-2} \to \ldots \to C^0]$. Hence $V$ has
projective dimension $\leq n+2$ by \eqref{gl.dim0}.

The last assertion holds since the $\C$-dimension of $M \in \Ac$
equals the projective dimension of $(\P_M)_\C$ in $\AC$.
\end{proof}

\subsection{Simple objects in $\AC$}\label{simple objects}
Recall that an abelian category has enough simples if each
indecomposable projective has a simple quotient. It was shown in
\cite{BJ} that $\AA$ has enough simples when $\Ac = \mod(R)$.
\begin{proposition}\label{simple objects prop}
\textbf{i)} Let $[X \to Y \to Z]$ be simple in $\AA$. Then the
$\C$-approximation $[X \to Y \to Z]_\C$ is zero or simple in $\AC$
and $Z \in \C \implies [X \to Y \to Z]_\C$ is non-zero.

\textbf{ii)} Assume that $\Ac$ is a Krull-Schmidt category and let
$Z \in \C$ be indecomposable. Then all simple quotients of $\P_Z$
in $\AC$ are isomorphic.

\textbf{iii)} If $\AA$ has enough simples  then $\AC$ has enough
simples. In fact, if $[X \to Y \to Z]$ is a simple quotient of
$\P_Z$ in $\AA$ then $[X \to Y \to Z]_\C$ is a simple quotient of
$\P_Z$ in $\AC$, for $Z \in \C$.  All simples of $\AC$ are of this
form.

\textbf{iv)} Assume that $\Ac$ is a Krull-Schmidt category and $V
= [X \to Y \to Z] \in \AA$ is simple with $Z$ indecomposable. Then
$V_\C = 0 \iff Z \notin \C$.
\end{proposition}
\begin{proof}  \textbf{i)} Write $V = [X \to Y \to Z]$. To show that $V_\C$ is simple or zero it is enough to show that any non-zero map $M \to V_\C$ in $\AC$ is surjective. By adjointness $M \to V_\C$ corresponds to a non-zero map $\taugA M \to V$ which is surjective since
$V$ is simple. Thus, $M \cong \pi_\C \taugA M \to V_\C$ is also
surjective since $\pi_\C$ is exact.

Assume now that $Z \in \C$. Then we can take $V_\C = [\ldots \to C^{-2} \to
C^{-1} \to Z]$,  $C^{i} \in \C$, and the map $V_\C \to V$ is $Id_Z$ in degree $0$. Since $V \neq
0$ we know that $Y \to Z$ is not a split epimorphism. It follows that
$C^{-1} \to Z$ can not be a split epimorphism. Hence $V_\C$ is
non-zero.

\medskip

\noindent \textbf{ii)} Let $L$ and $L'$ be simple quotients of $\P_Z$ and let $K = \Ker(\P_Z \to L)$ and $K' = \Ker(\P_Z \to L')$.
Then $\Jm (K \oplus K' \to \P_Z) \neq 0$ because $K^0
\oplus K'^0 \to Z$ cannot be a split epimorphism by the
Krull-Schmidt Theorem (since neither $K^0 \to Z$ nor $K'^0 \to Z$
are split epimorphisms). Since $\Jm (K \oplus K' \to \P_Z)$ is a quotient of both $L$ and $L'$ we conclude that $L \cong L'$.
\medskip

\noindent \textbf{iii)} This follows from \textbf{i)}.

\medskip

\noindent \textbf{iv)} Assume that $Z \notin \C$ and let $C \to Z$
be a $\C$-cover. Since neither this map nor $Y \to Z$ are split
the Krull-Schmidt Theorem gives that $Y \oplus C \to Z$ is not
split.  Thus, $[K \to Y \oplus C \to Z]$ (where $K = \Ker(Y \oplus
C \to Z)$) is a non-zero quotient object of $V$ and hence
isomorphic to $V$ since $V$ is simple. Thus $V_\C = 0$ by
Proposition \ref{Section AA prop} iv).

Conversely, assume that $Z \in \C$. Then we saw during the proof of Proposition \ref{Section
AA prop} iv) that $V \cong [K \to Y \oplus C \to Z]$
where $C \to Z$ is \emph{any} $\C$-cover. In particular we can take $C=Z$ and  then it follows that $V=0$ since $Y\oplus Z \to Z$ is a split surjection.
\end{proof}
\begin{definition}\label{Yonedasimples} If $\Ac$ is a Krull-Schmidt category and $\AA$ has enough simples we denote by $\L_Z$
the unique simple quotient of $\P_Z$ in $\AC$, for $Z \in \C$.
\end{definition}
To sum up we have showed
\begin{cor}\label{simple objects cor} Assume that $\Ac$ is Krull-Schmidt and that $\AA$ has enough simples. Then $\AC$ has enough simples,
each projective in $\AC$ has a unique simple quotient $\L_Z$ and
for $Z \in \Ac$ indecomposable we have $\pi_\C (\L^\Ac_Z) = \L_Z$,
if $Z \in \C$, and $\pi_\C (\L^\Ac_Z) = 0$ else. (Recall that $\L^\Ac_Z$ denotes the simple quotient of $\P_Z$ in $\Ac$.) Since $\AC$ has
enough projectives any simple object  is isomorphic to some
$\L_M$.
\end{cor}

\subsection{Yoneda embedding}\label{Yoneda embedding}
Let $\Mod \C$ be the category of all additive functors $\C^{op}
\to \operatorname{Ab}$, where $\operatorname{Ab}$ is the category
of abelian groups. This is an abelian category where for $\phi$ a
morphism in $\Mod \C$, $(\Ker \phi)(X) = \Ker \phi_X$ and $(\Coker
\phi)(X) = \Coker \phi_X$. Let $\mod \C$ be the full subcategory
of coherent functors, i.e. those functors $F$ which admit (projective) presentations
$$\Hom_\C(\- , X) \to \Hom_\C(\- ,Y) \to F \to 0,$$ with $X,Y \in
\C$.  It is well known that if $\C$ has pseudo-kernels then $\mod
\C$ is abelian and moreover the embedding $\mod \C \to \Mod \C$ is
exact (see \cite{AU}). Our category $\C$ is $\Ac$-approximating
and therefore do have pseudo-kernels because if $f:A \to B$ is a
morphism in $\C$ and $K$ its kernel in $\Ac$ then the composition
$K_\C \to K \hookrightarrow A$ is a pseudo-kernel of $f$.

The functor $\C \ni X \mapsto \Hom(\- ,X) \in \Proj(\mod \C)$ defines an
equivalence $\C \cong \Proj(\mod\C)$. We have functors between
triangulated categories
\begin{equation}\label{ffembed1}
\TC \isoto \T_{\Proj(\mod\C)} \isoto \D(\mod\C) \to \D(\Mod\C).
\end{equation}
These functors are $t$-exact with respect to $\TlC$ and the
tautological $t$-structures on the derived categories. Thus,
taking hearts, we get an equivalence followed by an exact fully
faithful embedding
$$
\rho_\C: \AC \isoto \mod \C \to \Mod\C.
$$
Let $X \in \C$ be indecomposable. If we assume that $\Ac$ is a
Krull-Schmidt category, then $\Hom_\C(\- , X)$ has a (unique)
maximal subobject $\J(\- ,X) \in \Mod \C$, called the radical. For
$Y \in \C$ we have $\J(Y, X) = \{f: Y \to X\mid f \hbox{ is not a
split epimorphism.}\}$. $\J(\- ,X)$ may not belong to $\mod \C$;
for instance, if $\C$ is abelian then $\J(\- ,X) \in \mod \C$ iff
there is an almost split short exact sequence ending with $X$, see
\cite{ARS}.

\begin{proposition}\label{Yoneda embedding prop}
\textbf{i)} For any $V \in \AC$ we have $V$ is simple $\iff$
$\rho_\C V$ is simple.

\textbf{ii)} Let $X \in \C$ be indecomposable. Assume that $\Ac$
is Krull-Schmidt and that $\P_X$ admits the (necessarily unique)
simple quotient $\L_X$ in \AC. Then $\rho_\C \L_X \cong \Hom_\C(\-
, X)/\J(\- ,X)$. Consequently, if  $\L_X =[\ldots \to C^{-n}
\overset{d^{-n}}{\longrightarrow} \ldots \to \C^{-1}
\overset{d^{-1}}{\longrightarrow} X]$ then  $(d^{-1})_*
\Hom_{\C}(\- , C^{-1}) = \J(\- , X)$ and
$$(*) \ \ \ \ \ \ \  \ldots \to \Hom_\C(\- , C^{-n}) \overset{(d^{-n})_*}{\longrightarrow} \ldots
\to \Hom_{\C}(\- , C^{-1}) \overset{(d^{-1})_*}{\longrightarrow}
\J(\- , X) \to 0 \hbox{ is exact. }
$$
\textbf{iii)} Conversely, if $[\ldots \to C^{-n}
\overset{d^{-n}}{\longrightarrow} \ldots \to \C^{-1}
\overset{d^{-1}}{\longrightarrow} X] \in \AC$  satisfies $(*)$
then this object is simple and thus isomorphic to $\L_X$.
\end{proposition}

\begin{proof}
\textbf{i)} Assume that $V$ is simple. Let $\phi: M \to \rho_\C V$
be a non-surjective map with $M \in \Mod \C$. We must show $\phi =
0$. Since $M$ is a quotient of some direct sum $\oplus_{i \in I}
\Hom_\C(\- , X_i)$, $X_i \in \C$, we may assume $M =  \oplus_{i
\in I} \Hom_\C(\-  , X_i)$. We have $M = \underrightarrow{\lim} \,
M_j$, where $M_j = \oplus_{i \in I_j}  \Hom_\C(\-  , X_{i}) =
\Hom_\C(\- , \oplus_{i \in I_j} X_{i})\in \mod \C$, for some
finite subset  $I_j \subset I$. But then  $\phi|_{M_j} : M_j \to
\rho_\C V$ is given by a non-surjective map $\oplus_{i \in I_j}
X_{i} \to V$ by the Yoneda lemma. The latter map is zero by the
simplicity of $V$. Thus $\phi|_{M_j} = 0$ and we conclude $\Jm
\phi \cong \Jm \underrightarrow{\lim} \, \phi|_{M_j} = 0$.

Conversely, if $\rho_\C V$ is simple then $V$ must be simple since
$\rho_\C$ is exact and fully faithful.

\smallskip

\noindent \textbf{ii)} By \textbf{i)} we have that $\rho_\C(\L_X)$
is a simple quotient of $\rho_\C(\P_X) = \Hom_\C(\- , X)$; hence
$\rho_\C(\L_X) = \Hom_\C(\-  , X)/\J(\-  ,X)$. This proves
exactness of $(*)$ in degree $0$. Exactness in all other degrees
follows from the $t$-exactness of the functors in
\eqref{ffembed1}.

\smallskip

\noindent \textbf{iii)} If $(*)$ holds we get $\rho_\C(\L_X) =
\Hom(\- , X)/\mathcal{J}(\- , X)$ is simple and thus isomorphic to
$\L_X$.
\end{proof}

\section{The structure of $\AC$ when $\C \subseteq \mod(R)$}\label{Duality Big Section}
In this section $\Ac = \mod(R)$ and $\C$ is an $\Ac$-approximating
full additive subcategory of $\Ac$. We deduce the existence of a
Serre functor $\SC$. Under the additional hypothesis that $\C$ is
functorially finite we show that $\AC$ has enough injectives and
deduce that $\AC$ is a quotient category of $\AA$.
\subsection{Serre duality and injectives}\label{Duality}  Recall from Section \ref{roosterlabel} that there is a
Serre functor $\SA: \Proj(\AA) \to \Inj(\AA)$ on $\AA$. This means
that there is a functorial isomorphism
\begin{equation}\label{serreextended}
\Dd \Hom_{\AA}(\P_M, \- )\cong \Hom_{\AA}(\- , \SA \P_M), \hbox{
for } M \in \Ac.
\end{equation}
Recall that $\Proj(\AC) \subseteq \Proj(\AA)$. We define a Serre
functor $\SC$ on $\AC$ by $\SC \df \pi_\C \circ \SA
|_{\Proj(\AC)}: \Proj(\AC) \to \AC$.
\begin{proposition}\label{Duality prop}   \textbf{i)}  We have $\Dd
\Hom_{\AC}(\P_M, \- ) \cong \Hom_{\AC}(  \- , \Ss_\C \P_M)$, for $M \in
\C$, and $\SC$ takes values in $\Inj(\AC)$.

\textbf{ii)} $\SC$ is fully faithful.

\textbf{iii)}  For $M \in \AC$ indecomposable $\SC \P_M$ is an
injective hull of $\L_M$.

\textbf{iv)} $\Ss_\C$ extends to a triangulated functor $\Ss_\C:
\TC \to \TC$.
\end{proposition}
\begin{proof} \textbf{i)} For $M \in \C$ and $V \in \AC$ we have
$$
\Dd \Hom_{\AC}(\P_M, V) =  \Dd \Hom_{\TA}(\P_M, V) \cong
$$
$$
\Hom_{\TA}(V, \Ss_\Ac \P_M) \cong \Hom_{\TC}(V, \Ss_\C \P_M) =
\Hom_{\AC}(V, \Ss_\C \P_M).
$$
Here, the second isomorphism holds by Lemma \ref{roosterlabel lem}
and the third isomorphism holds because $\pi_\C$ is right adjoint
to the inclusion $\TC \hookrightarrow \TA$. Since $\P_M$ is
projective it follows that $\Hom_{\AC}(\- , \Ss_\C \P_M)$ is exact
so that $\SC \P_M$ is injective.
\medskip

\noindent \textbf{ii)}   For $M, M' \in \C$ we have isomorphisms
$$
\Hom_{\AC}(\P_M, \P_{M'}) \isoto \Dd \Hom_{\AC}(\P_{M'}, \Ss_\C
\P_{M}) \isoto \Hom_{\AC}(\Ss_\C \P_{M}, \Ss_\C \P_{M'})
$$
which proves that $\SC$ is fully faithful.

\medskip

\noindent \textbf{iii)}  Since $\SC$ is fully faithful and $\P_M$
is indecomposable we conclude that $\SC \P_M$ is indecomposable.
Moreover, since $\Hom_{\AC}(\L_M, \SC \P_M) \cong \Dd
\Hom_{\AC}(\P_M, \L_M) \neq 0$ we conclude that $\L_M$ embeds to
$\SC \P_M$.

\medskip

\noindent \textbf{iv)} The extension is given as follows (compare
with \cite{BJ}, Proposition 4.6)
$$
\TC \cong \T_{\Proj(\AC)} \overset{\Ss_\C}{\longrightarrow}
\T_{\Inj(\AC)} \to \T_{\AC} \to \D(\AC) \cong \T_{\Proj(\AC)}
\cong \TC.
$$
\end{proof}
Let $C \in \C$ be indecomposable. Then $\End_{\AC}(\P_C) =
\End_\C(C)$ is a local artinian ring. We let $\m$ be its maximal
ideal. Now the simple quotient $\L_C$ of $\P_C$ can be
characterized as follows
\begin{cor}\label{Duality cor} Let $\tau \in \Hom_{\AC}(\P_C, \Ss \P_C)$ be non-zero such that $\tau \circ \m = 0$. Then $\L_C \cong \Jm \tau$.
(The uniqueness of $\L_C$ implies that such a $\tau$ is unique up
to a scalar.)
\end{cor}
\begin{proof}
Since $\Hom_{\AC}(\P_C ,\Ss \P_C) \cong \Dd \End_{\AC}(\P_C) \neq
0$ we can find $0 \neq \tau \in \Hom_{\AC}(\P_C ,\Ss \P_C)$ with
$\tau \circ \m = 0$.

We show $\Jm \tau$ is simple. Let $i: L' \hookrightarrow \Jm \tau$
be an monomorphism which is not surjective. We must show $L' = 0$.
This follows if we can show $\Hom_{\AC}(L',\Ss \P_C) = 0$ which in
turn  amounts to show that $\Hom_{\AC}(\P_C,L') = 0$.

Let $f \in \Hom_{\AC}(\P_C,L')$. Since $\P_C$ is projective $i
\circ f = \tau \circ h$ for some $h \in \End_{\AC}(\P_C)$. Since
$i \circ f$ is not surjective we get that $h$ is a non-unit. Thus
$\tau \circ h = 0$. Thus $f = 0$.
\end{proof}

\subsection{{}}\label{injectivesec}
\begin{proposition}\label{injectivesec prop} Assume that $\C$ is functorially finite. Then $\AC$ has enough injectives.\footnote{Proposition \ref{injectivesec prop}  is not optimal in the sense that it frequently happens that $\AC$ has enough injectives even when
$\C$  is not functorially finite. E.g. if $\C = \Proj(\mod(R))$,
then $\AC \cong \mod(R)$ has enough
injectives.}\label{gallo} They are of the form $\Ss_\C P_{C}$ for
$C \in \C$. In particular, each injective in $\AC$ is the
$\C$-approximation of some injective in $\AA$.
\end{proposition}
\begin{proof}  Let $M \in \AC$. We shall construct an embedding of $M$ into an injective object of $\AC$. Let $\Ss_\Ac \P_X$ be an
injective hull of $\taug_\Ac M$ in $\AA$. Let $\Dd C
\twoheadrightarrow \Dd X$ be a $\Dd \C$-cover of $\Dd X$ in $\Dd
\Ac$ and let $X \hookrightarrow C$ be the dual map (i.e. a
``$\C$-hull"). By Proposition \ref{Duality prop} $\Ss_\C \P_{C}
\in \Inj(\AC)$. We have
$$M = (\taug_\Ac M)_\C \hookrightarrow \Ss_\C \P_X \to  \Ss_\C
P_{C}.$$ Thus it remains to be shown that $\Ss_\C \P_X \to  \Ss_\C
\P_{C}$ is injective. This follows if we can show that $\Hom_{\TC}(V, \Ss_\C \P_X) \to \Hom_{\TC}(V, \Ss_\C \P_{C})$ is injective, for
$V \in \AC$. By adjointness this amounts to show that
\begin{equation}\label{firstcheck}
\Hom_{\TA}(V, \Ss_\Ac \P_X) \to \Hom_{\TA}(V, \Ss_\Ac \P_{C})
\hbox{ is injective for }  V \in \AC.
\end{equation}
Note that $\Dd V \in \operatorname{K}^+(\Dd \C)$, where $\Dd V \df
\Hom^\bullet_k(V,E)$ is the dual complex, and that  the natural
map
$$\Hom_{\operatorname{K}^+(\Dd \Ac)}(\Dd V, \P_{\Dd C}) \to \Hom_{\operatorname{K}^+(\Dd \Ac)}(\Dd V, \P_{\Dd X})$$
is surjective since $\Dd C \to \Dd X$ is a $\C$-cover. Thus
\eqref{firstcheck} follows since we now have
$$\Hom_{\TA}(V, \Ss_\Ac \P_X) = \Dd \Hom_{\TA}(\P_X, V) = \Dd \Hom_{\operatorname{K}^+(\Dd \Ac)}(\Dd V, \P_{\Dd X}) \hookrightarrow$$
$$
\Dd \Hom_{\operatorname{K}^+(\Dd \Ac)}(\Dd V, \P_{\Dd C}) =  \Dd
\Hom_{\TA}(\P_{C}, V) = \Hom_{\TA}(V, \Ss_\Ac \P_{C}).
$$
\end{proof}

\begin{cor}\label{injectivesec cor} Assume that $\C$ is functorially finite. Then   $\SC: \Proj(\AC) \to \Inj(\AC)$ is an equivalence of categories.
Moreover, if we assume in addition that $\Ac$ has finite
$\C$-dimension then the extension $\Ss_\C: \TC\to \TC$ restricts
to an auto-equivalence $\Ss_\C: \T^b_{\C} \to \T^b_{\C}$ of
triangulated categories.
\end{cor}
\begin{proof} That  $\Ss_\C: \Proj(\AC) \to \Inj(\AC)$ is an equivalence follows from Proposition \ref{Duality prop} and Proposition \ref{injectivesec prop}.

The assumption that $\Ac$ has finite $\C$-dimension is equivalent
to assume that $\AC$ has finite cohomological dimension by
Corollary \ref{Projectives in AC cor}. Thus $\Ss_\C: \T^b_{\C} \to
\T^b_{\C}$ is given by the composition
$$
\T^b_{\C} \cong \T^b_{\Proj(\AC)} \overset{\Ss_\C}{\isoto}
\T^b_{\Inj(\AC)} \cong \Db(\AC) \cong \T^b_\C.
$$
\end{proof}
\begin{theorem}\label{injectivesec thm}  Assume that $\C$ is functorially finite. Then $\pi_\C : \AA \to \AC$ is a quotient functor, i.e. $\pi_\C$ has a right adjoint $\sigma$ such that the adjunction morphism $\pi_\C \circ \sigma \to Id_{\AC}$ is an isomorphism.
\end{theorem}
\begin{proof}
Let $\Ind(\AA)$ be the abelian category of ind-objects in $\AA$
(see \cite{KS}). Then $\pi_\C$ extends to an exact functor
$\pi_\C: \Ind(\AA) \to \Ind(\AC)$ which commutes with direct sums.
Since $\Ind(\AA)$ is a Grothendieck category $\pi_\C$ has a right
adjoint $\sigma: \Ind(\AC) \to \Ind(\AA)$.

\smallskip

We show that $\sigma$ restricts to an adjoint of $\pi_\C: \AA \to
\AC$, i.e. that $\sigma V \in \AA$ for $V \in \AC$. Pick an
injective corepresentation  $0 \to V \to \SC \P_{C^0}
\overset{f}{\longrightarrow} \SC \P_{C^1}$. Since $\sigma$ is left
exact (being a right adjoint) we have $\sigma V \cong \Ker
\sigma(f)$. On the other hand, to show that $\Ker \sigma(f) \in
\AA$ it is enough to show that $\sigma \Ss_\C \P_{C} \in \AA$ for
$C = C^{0}, C^1$. This holds since we will actually show
\begin{equation}\label{sSeq}
\sigma \Ss_\C \P_C \cong \Ss_\Ac \P_C.
\end{equation}
Since $\AA$ is dense in $\Ind(\AA)$, in order to establish
\eqref{sSeq}, it suffices to note that for $W \in \AA$ we have
$$
\Hom_{\Ind(\AA)}(W, \sigma \Ss_\C \P_C) \cong \Hom_{\AC}(\pi_\C W,
\Ss_\C \P_C) \cong \Dd \Hom_{\AC}(\P_C, \pi_\C W)  \cong
$$
$$
\Dd  \Hom_{\AA}(\P_C, W)  \cong  \Hom_{\AA}(W, \Ss_\Ac \P_C).
$$
We show that $\pi_\C \circ \sigma \to Id_{\AC}$ is an isomorphism.
Since $\pi_\C \circ \sigma$ is left exact and $\AC$ has enough
injectives it is enough to verify that $\pi_\C \circ \sigma
(\Ss_\C \P_{C}) \to \Ss_\C \P_{C}$ is an isomorphism. This again
holds since $\sigma \Ss_\C \P_{C}  \cong \Ss_\Ac \P_C$.
\end{proof}
\section{The structure of $\AC$ when $\C$ is a maximal $n$-orthogonal subcategory}\label{Oyama's own section}
The material in this section is a condensed version of the master
thesis  \cite{A} of the first author. We keep the assumptions on
$\Ac = \mod(R)$ and $\C$ from Section \ref{Duality Big Section}.
Following \cite{I} we study the case where $\C$ is a maximal
$n$-orthogonal subcategory of $\Ac$. We prove that simples of
$\AC$ are generalized AR sequences and reprove the main results of
\cite{I} with new methods. A crucial point is that injectives in
$\AC$ take a very specific form (Proposition \ref{again my sweet
injectives prop}) which allows us to establish the AR duality
(Theorem \ref{really really AR thm}).

\subsection{{}}\label{def n ort sec}
Let $n \geq 0$. Following \cite{I} we call $\C$ a maximal $n$-orthogonal subcategory if it is functorially finite,
$\Proj(\Ac), \Inj(\Ac) \subseteq \C$ and for an object $X \in \Ac$ we have
$$X \in \C \iff  \Ext^i(\C, X) = \Ext^i(X,\C) = 0, \, \forall \, 1 \leq i <n.$$
(For various examples of such subcategories see \cite{I}.) We
assume throughout Section \ref{Oyama's own section} that $\C$ is
maximal $n$-orthogonal.

\begin{lem}\label{def n ort sec lem} [see \cite{I}] \textbf{i)} Let
$0 \to X^{-n-2}\overset{d^{-n-2}}{\longrightarrow} \ldots \to
X^{-1} \overset{d^{-1}}{\longrightarrow} X^0 \to 0$ be an exact
sequence with terms in $\C$. Then the following conditions are
equivalent $\operatorname{(1)}$ $d^{-1}$ is split $\iff$
$\operatorname{(2)}$ $d^{-n-2}$ is split $\iff$
$\operatorname{(3)}$  the sequence is homotopic to $0$.

\textbf{ii)} Any $X \in \Ac$ has $\C$-dimension $\leq n$.
\end{lem}
\begin{proof}
$\textbf{i)}$. $\operatorname{(1)} \implies \operatorname{(2)}$.
Assume that $\operatorname{(2)}$ is false and let $1 < m < n+1$ be
the smallest integer such that $X^{-m} \to \Jm d^{-m}$ doesn't
split. Thus we have the exact sequence $0 \to X^{-n-2} \to \ldots
\to X^{-m} \to \Jm d^{-m} \to 0$ with terms in $\C$ (since $\Jm
d^{-m}$ is a direct summand of $X^{-m+1}$ it belongs to $\C$). A
simple d\'evissage shows that $ \Ext^{n+1-m}_{\Ac}(\Jm d^{-m},
X^{-n-2}) \cong \Ext^1_\Ac(\Jm d^{-m}, \Jm d^{{-m-1}}) \neq 0$.
This contradicts that $\Jm d^{-m}, X^{-n-2} \in \C$. The other
implications are proved the same way.

\medskip

\noindent \textbf{ii)} Let $C^{-n-1}
\overset{d^{-n-1}}{\longrightarrow} \ldots \to C^{-1}
\overset{d^{-1}}{\longrightarrow} C^0
\overset{d^0}{\twoheadrightarrow} X$ be the beginning of a
$\C$-resolution of $X$. Let $K^{i}  = \Ker d^i$, $K^1 = X$ and $Y
\in \C$. We must prove $K^{-n-1} \in \C$. Applying $R\Hom_\Ac(\- ,
Y)$ to the short exact sequences $K^i \into C^i \tto K^{i+1}$ we
get by induction that $\Ext^j_\Ac(K^{-n-1}, Y) = 0$ for all $1\leq
j \leq n$. Thus $K^{-n+1} \in \C$.
\end{proof}
\begin{proposition}\label{def n ort prop} The set of objects in $\AC$ are precisely
$ \Lambda := \{V \in \TC\mid  V \cong [C^{-n-2} \to \ldots \to
C^1 \to C^0],\, C^i \in \C, \, H^i(V) = 0 \hbox{ for } i < 0\}. $
\end{proposition}
\begin{proof} Let $V \in \AC$. By Corollary \ref{Projectives in
AC cor} and ii) of the previous lemma we obtain that  $\AC$ has
cohomological dimension $\leq n+2$. By \eqref{gl.dim0} this
implies that $V \cong [C^{-n-2} \to \ldots \to C^{-1} \to C^0] \in
\Lambda$. We also note for the record that it follows that $V$
admits the projective resolution
\begin{equation}\label{gl.dim}
0 \to \P_{C^{-n-2}} \to \ldots \to \P_{C^{-1}} \to \P_{C^0} \to V
\to 0.
\end{equation}

\noindent Conversely, assume that $V = [C^{-n-2}
\overset{d^{-n-2}}{\longrightarrow} \ldots \to C^{-1}
\overset{d^{-1}}{\longrightarrow} C^0] \in \Lambda$. Then
evidently $V \in \TlC$. To prove that $V \in \TgC$ we must show
that $\Hom_{\TC}(\T^{\leq -1}_\C,V) = 0$. For this it suffices to
show that $\Hom_{\TC}(X[i],V) = 0$, for $X \in \C$ and $i>0$. This
in turn holds because
$$
\Hom_{\TC}(X[i],V) = \Hom_\Ac(X, \Jm d^{-i-1})/(d^{-i-1})_*
\Hom_\Ac(X, C^{-i-1})$$ is isomorphic to a submodule of
$\Ext^1_\Ac(X, \Jm d^{-i-2})$ and the latter vanishes because
$n$-orthogonality implies
$$
\Ext^1_\Ac(X, \Jm d^{-i-2}) \cong \Ext^2_\Ac(X, \Jm d^{-i-3})
\cong \ldots \cong \Ext^{n-i+1}_\Ac(X, C^{-n-2}) = 0.
$$
\end{proof}

\subsection{Injectives}\label{again my sweet injectives}
\begin{proposition}\label{again my sweet injectives prop} Let $\C$ be a maximal $n$-orthogonal subcategory. Then the injectives of $\AC$ are precisely of the form $[X \to I^{-n-1} \to \ldots \to I^0]$
for $X \in \C$ and $I^j \in \Inj(\mod(R))$.
\end{proposition}
\begin{proof}
Let $J = [X \to I^{-n-1} \to \ldots \to I^0] \in \AC$ and $U \in
\AC$ be arbitrary. Then by \eqref{gl.dim} we have
$$\Ext^1_{\AC}(U,J) \cong \Ext^{n+3}_\Ac(U,\P_X) =0.$$ This shows
that $J$ is injective. Since retracts of $J$ are clearly of the
same form as $J$ it now suffices to show that any object $V =
[C^{-n-2}\to C^{-n-1} \to \ldots \to C^0]  \in \AC$ embeds to an
object of the same form as $J$. Let $J' = [C^{-n-2} \to I'^{-n-1}
\to \ldots \to I'^0]$ be a part of an injective resolution of
$C^{-n-2}$ in $\Ac$. The identity map on $C^{-n-2}$ lifts to a map
$\phi: V \to J'$
that we shall prove is injective.

Adding an injective summand to $I'^0$ if necessary we can assume
that the induced map $H^0(V) \to H^0(J')$ is injective. The last
condition is equivalent to  $C^{-1} \times I'^{-2} \to C^0
\times_{I'^0} I'^{-1}$ being surjective (see \cite{BJ} Lemma 3.1).
We have $$\Ker \phi = [0 \to C^{-n-1} \to C^{-n} \times I'^{-n-1}
\to \ldots \to C^{-1} \times I'^{-2} \to C^0 \times_{I'^0}
I'^{-1}]_\C.$$
Since $0 \to C^{-n-1}$ is split it follows from
Lemma \ref{def n ort sec lem} i) that $\Ker \phi = 0$.
\end{proof}
\subsection{AR sequences}\label{AR seq sec againnn}
By Corollary \ref{simple objects cor} we know that the simples of
$\AC$ are precisely the simple quotients $\L_X$ of $\P_X$, for
$X \in \C$ indecomposable. An AR sequence in $\C$ is by definition
a simple object $\L_X \in \AC$ such that $X$ is non-projective.

According to  \cite{I} an $(n+1)$-almost split sequence in $\C$ is
an exact sequence
\begin{equation}\label{ARsequencedef}
0 \to X' \to C^{-n} \to \ldots \to C^{-1} \to C^0 \to X \to 0
\end{equation}
with terms in $\C$ with $X$ indecomposable and non-projective
which induces an exact sequence
$$
0 \to \Hom_\C(\- ,X') \to\Hom_\C(\- , C^{-n}) \to \ldots \to
\Hom_{\C}(\- , C^{0}) \to \mathcal{J}(\- , X) \to 0.
$$
Let $V = [X' \to C^{-n} \to \ldots \to C^{-1} \to C^0 \to X] \in
\AC$.
\begin{theorem}\label{again my sweet injectives thm} $V$ is a simple object of $\AC$; hence $V \cong \L_X$. Conversely, since $\L_X$  is simple
it defines an almost split sequence.
\end{theorem}
\begin{proof}This follows from Proposition \ref{Yoneda embedding prop}.
\end{proof}
This implies again that almost split sequences with ending term
$X$ are unique up to homotopy.
\subsection{AR duality}\label{really really AR}
Let $X,Y \in \C$.  Pick a part of a projective resolution $P^{-n}
\to \ldots \to P^0 \tto Y$ in $\Ac$ and let $K = \Ker (P^{-n} \to
P^{-n+1})$. Put $V = [K \to P^{-n} \to \ldots \to P^0 \to Y] \in
\TA$.  We have $\Ss_\C \P_X = [X' \to I^{-n-1} \to \ldots \to
I^0]$ where the $I^j$'s are injective. Then $\Hom_{\TC}(\P_X,V_\C)
\cong\Hom_{\TA}(\P_X,V) \cong \underline{\Hom}_\Ac(X,Y)$, where
$\underline{\Hom}_\Ac$ is stable hom, see \cite{ARS}.  On the
other hand
$$\Hom_{\TC}(V_\C, \Ss_\C \P_X) \cong \Hom_{\TA}(V, \Ss_\C \P_X) \cong
$$
$$
\Hom_{\TA}(V, X'[n+2]) \cong \Ext^{n+1}_\Ac(Y,X').
$$
Here, the first isomorphism holds since  $\Ss_\C \P_X$ starts in
position $-2-n$. To establish the second last isomorphism we used
the fact that any map $K \to X'$ lifts to a map $V \to \Ss\P_X$
which is unique up to homotopy since the $I^j$'s are injective and
$V$ is acyclic. Thus we have proved
\begin{theorem}\label{really really AR thm} $\underline{\Hom}_\Ac(X,Y) \cong  \Dd \Ext^{n+1}_\Ac(Y,X')$.
\end{theorem} This formula was proved in \cite{I} with $X'$ replaced by $\DTr \Omega^n X$ where we recall that $\Omega^n X$ denotes the $n$'th syzygy in a minimal projective resolution of $X$
($\Omega^0 X = X$.) Thus $X' \cong \DTr \Omega^n X$.

\section{Frobenius algebras}\label{section frobenius algebras}
In this section $R$ is a finite dimensional  commutative Frobenius
algebra over a field $k$ and $\Ac = \mod(R)$. $\C \subseteq \Ac$
is a full additive subcategory satisfying $\Dd \C = \C$ and $R \in
\C$. (Here, $\Dd \C \subseteq \mod(R^{op}) = \Ac$ since $R$ is
commutative. See definition \ref{big t Section defi2}.)

\subsection{Frobenius duality}\label{Frobenius algebras section}
The assumption that $R$ is Frobenius implies that $R$ is self
injective. Thus $\Dd = {}^* = \Hom_k( \- , k) = \Hom_R( \- , R)$ in this case.
All projectives in $\Ac$ are self dual; in particular, $\Proj(\Ac) = \Inj(\Ac)$.
This leads to a nice duality theory on $\AA$ and $\AC$.

Let $M \in \Ac$. Assume first that $M$ has no projective direct
summand. Pick projective resolutions
\begin{equation}\label{Frobenius1}
P^{-1} \overset{\pa}{\longrightarrow} P^0 \to M \to 0 \hbox{ and }
P'^{-1} \overset{\pa'}{\longrightarrow} P'^0 \to M^* \to 0
\end{equation}
We obtain
\begin{equation}\label{Frobenius2}\Dd \Tr M \df \Dd (\Coker \pa^*) = (\Coker \pa^*)^* \cong \Ker \pa.\end{equation}
This implies that $\Ker \pa \into P^{-1} \to P^0$ is an injective
copresentation of $\Dd \Tr M \cong \Ker \pa$. Hence we have by
property $(8)$ in Section \ref{roosterlabel} that
\begin{equation}\label{Frobenius3}
\SA \P_M \cong [\Ker \partial \to P^{-1} \to P^0].
\end{equation}
If $M$ is projective the same property shows that  $\SA \P_M \cong
\P_M$.

\medskip
A duality functor on a category is a contravariant autoequivalence
whose square is isomorphic to the identity.
In \cite{BJ} a duality functor (here denoted) $d_\Ac: \AA \to \AA$
was defined as follows: If $M$ has no projective direct summand we
first define $d_\Ac \P_M = [{\Ker \pa'} \to P'^{-1} \to P'^0]$ and
if $M$ is projective $d_\Ac \P_M \df \P_M$. Taking direct sums we
have defined $d_\Ac \P_M$ for an arbitrary $M$. For a general
object $V = [K \to L \to M] \in \AA$ we define
\begin{equation}\label{Frobenius4-1}
d_\Ac V = \Ker(d_\Ac \P_M \to d_\Ac \P_L).
\end{equation}
It is straightforward to verify that $d_\Ac$ defines a
contravariant functor satisfying $d_\Ac \circ d_\Ac \cong
Id_{\AA}$.  By \eqref{Frobenius3} and item (8)
in Section \ref{roosterlabel} the functor $d_\Ac$ relates
with the Serre duality functor as follows:
\begin{equation}\label{Frobenius4}
d_\Ac \P_M = \SA \P_{M^*}, \hbox{ for } M \in \Ac.
\end{equation}

\begin{lem}\label{Frobenius algebras section lem} We have  $d_\Ac \Ker(\pi_\C|_{\AA}) = \Ker(\pi_\C|_{\AA})$.
\end{lem}
\begin{proof} Put $\mathcal{K} = \Ker(\pi_\C|_{\AA})$ and let $V \in
\mathcal{K}$. The following is equivalent: $d_\Ac V \in
\mathcal{K} \iff$
$$
\Hom_{\AC}(W,\pi_\C d_\Ac V) = 0, \forall \, W \in \AC \iff
\Hom_{\AC}(\P_C,\pi_\C d_\Ac V) = 0, \forall \, C \in \C.
$$
Here the last equivalence holds by Proposition \ref{Projectives in
AC prop}. We have
$$
\Hom_{\AC}(\P_C,\pi_\C d_\Ac V) \cong \Hom_{\AA}(\taugA \P_C,
d_\Ac V) \cong \Hom_{\AA}(\P_C, d_\Ac V) \cong
$$
$$
\Hom_{\AA}(V , d_\Ac \P_C) \cong \Hom_{\AA}(V , \SA \P_{C^*})
\cong \Dd \Hom_{\AA}(\P_{C^*}, V) \cong
$$
$$
\Dd\Hom_{\AA}(\taugA \P_{C^*}, V) \cong \Dd \Hom_{\AC}(\P_{C^*},
\pi_\C V) \cong  \Dd \Hom_{\AC}(\P_{C^*}, 0)= 0.
$$
Here we used that $\taugA \P_C = \P_C$, the adjointness $(\taugA:
\AC \leftrightarrows \AA : \pi_\C)$, Serre duality (see item (7)
in Section \ref{roosterlabel}) and \eqref{Frobenius4}. Hence
$d_\Ac \mathcal{K} \subseteq \mathcal{K}$. Since $d^2_\Ac =
Id_{\AA}$ we conclude $d_\Ac\mathcal{K} = \mathcal{K}$.
\end{proof}
\begin{defi}\label{Frobenius algebras section defi}
Define the functor $d_\C \df \pi_\C \circ d_\Ac \circ \taug_\Ac:
\AC \to \AC$.
\end{defi}
\begin{proposition}\label{Frobenius algebras section prop} $d_\C$ is a duality functor and $d_\C \circ \pi_\C \cong
\pi_\C \circ d_\Ac$.
\end{proposition}
\begin{proof}
Let $\mu: \taugA \circ \pi_\C|_{\AA} \to Id_{\AA}$ be the
adjunction morphism and let $V \in \AA$. We get the exact sequence
\begin{equation}\label{ewewewewew}
0 \to \Ker \mu_V \to \taugA \pi_\C(V)
\overset{\mu_V}{\longrightarrow} V \to \Coker \mu_V \to 0.
\end{equation}
Since the adjunction morphism $Id_{\AC} \to \pi_\C \circ \taugA$
is an isomorphism (Proposition \ref{Section AA prop}) we get after
applying $\pi_\C$ to \eqref{ewewewewew} the isomorphism $\pi_\C
\taugA \pi_\C(V) \isoto \pi_\C V$ and we conclude that $\pi_\C
\Ker \mu_V = \pi_\C \Coker \mu_V = 0$. By Lemma \ref{Frobenius
algebras section lem} it follows that
 $\pi_\C d_\Ac \Ker \mu_V = \pi_\C d_\Ac \Coker \mu_V =
0$. Therefore, applying $\pi_\C \circ d_\Ac$ to \ref{ewewewewew}
gives an isomorphism
$$
\pi_\C d_\Ac \taugA \pi_\C(V) \isoto \pi_\C d_\Ac V.
$$
This isomorphism is functorial in $V$ and hence (by definition of
$d_\C$) gives the asserted isomorphism of functors $d_\C \pi_\C
\isoto \pi_\C d_\Ac$. We now get $$d^2_\C = (d_\C
\pi_\C)d_\Ac\taugA \cong (\pi_\C d_\Ac)d_\Ac\taugA \cong
Id_{\AC}$$ since $d^2_\Ac \cong Id_{\AA}$ and $\pi_\C \taug_\Ac
\cong Id_{\AC}$.
\end{proof}

\begin{cor}\label{Frobenius algebras section cor} \textbf{i)} For $C \in \C$ we have  $d_\C \P_C \cong \SC \P_{C^*}$. In particular, if $C$ is self dual we have $d_\C \P_C \cong \SC \P_C$.
 \textbf{ii)} If $C, C' \in \C$ are self dual, then $\Hom_{\AC}(\P_C, \P_{C'}) \cong  \Hom_{\AC}(\P_{C'}, \P_{C})$.
\end{cor}
\begin{proof} \textbf{i)} follows from \eqref{Frobenius4}. For
\textbf{ii)} we have
$$\Hom_{\AC}(\P_C, \P_{C'}) \cong \Dd \Hom_{\AC}(\P_{C'}, \SC \P_{C}) \cong \Hom_{\AC}(\SC \P_{C}, \SC \P_{C'}) \cong
$$
$$
\Hom_{\AC}(d_\C \P_{C}, d_\C \P_{C'}) \cong \Hom_{\AC}(\P_{C'},
\P_{C}).
$$
\end{proof}
\subsection{{}}\label{it is a t structure on O} Under the hypothesis of Section \ref{section frobenius algebras} we can improve Proposition \ref{partial converse sds prop} iii) as follows:
\begin{proposition}\label{it is a t structure on O prop} Assume that $\TlC$ is a $t$-structure on $\TC$. Then $\C$ is $\Ac$-approximating and functorially finite.
\end{proposition}
\begin{proof} Since $R$ is Frobenius any $M \in \Ac$ is isomorphic to $\Ker f$ for some morphism $R^n \overset{f}{\longrightarrow} R^m$, i.e. $M$ is the kernel of a $\C$-morphism.
The result now follows from Proposition \ref{partial converse sds
prop} iii) and the fact that $\C = \Dd \C$.
\end{proof}

\section{Category $\BGG$}\label{section cat bgg}
We apply our theory to the subcategory of projectives in the
Bernstein-Gelfand-Gelfand category $\BGG$ of a semi-simple complex
Lie algebra $\g$. By theory of Soergel, \cite{S}, this can be
studied by means of a certain subcategory of modules over the
coinvariant algebra $R$ of the Weyl group. The coinvariant algebra
is Frobenius so the results of Section \ref{section frobenius
algebras} apply. In Section \ref{Soergel's theory} we review
Soergel's theory and conclude that category $\BGG$ this ways fits
into higher AR theory. We also discuss the Rouquier complex (see
\cite{EW}) in this context.

\subsection{Category $\BGG$, coinvariant algebra and higher AR theory}\label{Soergel's theory}
Let $\g$ be a complex semi-simple Lie algebra, $\mathfrak{h}
\subset \b \subset \g$ a Cartan subalgebra contained in a Borel
subalgebra.  Let $W$ be the Weyl group and $R \df
S(\h)/(S(\h)^W_+)$ the coinvariant algebra. Here $S(\h)^W_+$ are
the $W$-invariant polynomial functions on $\h^*$ without constant term. For a simple
reflection $s$ let $R^s$ be the invariant ring and $\alpha_s \in
\h^*$ the corresponding simple root. Note that $R$ is free over
$R^s$ with basis $1, \alpha_s$. We write $\T_R = \T_{\mod{(R)}}$, $\AR = \A_{\mod(R)}$ and $d_R = d_{\mod(R)}$, etc.

Let $\BGG$  be the Bernstein-Gelfand-Gelfand category of
representations of $\g$, see \cite{H}. For the sake of simplicity
we shall restrict our attention to the so called principal block
$\BGG_0$ consisting of modules with trivial generalized central
character. Other blocks can be handled with similar methods.

There is the standard duality functor $d_\BGG: \BGG_0 \to \BGG_0$
given by taking the direct sum of all the dual $\h$-weight spaces
in a module; $d_\BGG$ fixes the simples. Let $w_0 \in W$ be the
longest element. Let $M_x$ denote the Verma module
with highest weight $x\rho-\rho$, for $x \in W$, where $\rho$ is
half the sum of the positive roots. Let $P_x$ be a
projective cover of $M_x$.
In \cite{S} a functor $\V: \BGG_0 \to \mod(R)$ with the following
properties was constructed:
\begin{enumerate}
\item $\V$ is exact. \item $\V P_{w_0} \cong R$. \item $\V M_x$ is
isomorphic to the trivial $R$-module $\k$, for $x \in W$. \item
$\V |_{\Proj{\BGG_0}}$ is fully faithful. \item $\V \circ d_\BGG
\cong {}^* \circ \V$ and $\V P \cong (\V P)^*$ for $P \in
\Proj(\BGG_0)$.
\end{enumerate}
For a simple reflection $s$ define an $R$-bimodule $B_s = R
\ot_{R^s} R$. For $\ul{x} = (s_m, \ldots , s_1)$ a sequence of
simple reflections put $x = s_m \cdots s_1 \in W$.  Define the
\emph{Soergel module} $B_x = \V(P_x)$ and the \emph{Bott-Samuelson
module}
$$\BSx = B_{s_m} \ot_R \ldots \ot_R B_{s_1} \ot_R \mathbb{C} \in \mod(R).$$
Any Bott-Samelson module splits into a direct sum of Soergel
modules and if $\ul{x}$ is reduced then $\BSx \cong B_x \bigoplus
\oplus_{y < x} B^{a_{y,x}}_y$, for some numbers $a_{y,x}$ (see
\cite{S2}). Consider the full subcategory $\B$ of $\mod(R)$ whose
objects are isomorphic to direct sums of Soergel modules. Then we
have an equivalence
$$
\V: \Proj(\BGG_0) \isoto \B
$$
From this and the fact that $\BGG$ has finite
cohomological dimension we get equivalences
$$
 \Db(\BGG_0) \cong \Kb(\Proj(\BGG_0)) \isoto   \TB .
$$
The tautological $t$-structure on $\Db(\BGG_0)$ therefore
corresponds to the $t$-structure $\TlB$ on $\TB$ which is the
strictly full subcategory generated by $\{X\in \TB\mid X^i = 0
\hbox{ for } i>0\}$.\footnote{It would be interesting to know if
$\TlB$ is a $t$-structure for non-crystallographic Coxeter groups
$W$ as the heart of this would serve as a category $\BGG$ for a
non-existing Lie algebra. Compare with \cite{EW} Remark 6.2.} We
get the induced equivalence
$$\A_\V: \BGG_0 \isoto \AB \df \TlB \cap \TgB$$ on hearts.

Category $\BGG$ fits into the framework of (Frobenius) higher AR
theory as follows:
\begin{theorem}\label{Soergel's theory thm} \textbf{i)} $\B$ is $R$-approximating, functorially finite and $R$ has finite $\B$-dimension.

\textbf{ii)} $\BGG_0$ is equivalent to a quotient category of
$\AR$.

\textbf{iii)} There is a Serre functor $\SO: \Proj(\BGG_0) \isoto
\Inj(\BGG_0)$ with the property that
$$\Hom_{\BGG_0}(P, M) \cong \Dd \Hom_{\BGG_0}(M, \SO P)$$ for
$P \in \Proj(\BGG_0)$ and $M \in \BGG_0$. It extends to an
auto-equivalence of triangulated categories $\SO: \Db(\BGG_0)
\isoto \Db(\BGG_0)$.

\textbf{iv)} We have $\SO P \isoto d_{\BGG} P$ for $P \in
\Proj(\BGG_0)$.

\textbf{v)} Any $I \in \Inj(\BGG_0)$ admits a projective
resolution of the form $P^a_{w_0} \to P^b_{w_0} \to I \to 0$.

\textbf{vi)} $d_{\B} \circ \A_\V \cong \A_\V \circ d_\BGG$, where
$d_\B$ is the duality functor on $\AB$ given by Definition
\ref{Frobenius algebras section defi}.

 \textbf{vii)} $\Hom_{\BGG_0}(P,P') \cong
\Hom_{\BGG_0}(P',P)$, for $P,P' \in \Proj(\BGG_0)$.
\end{theorem}
\begin{proof} \textbf{i)} The first part follows from Proposition \ref{it is a t structure on O prop} since $\TlB$ is a $t$-structure.
$R$ has finite $\B$-dimension since $\BGG_0$ has finite
cohomological dimension.

 \medskip

\noindent \textbf{ii)} This follows from Theorem \ref{injectivesec thm}.

 \medskip

\noindent \textbf{iii)} $\SO$ is defined by transporting $\Ss_\B$ by means of $\A_\V$. The last assertion follows from Corollary \ref{injectivesec cor}.

 \medskip

\noindent \textbf{iv)} We can assume that $P$ is indecomposable.
Thus $P$ is a projective cover of some simple module $L \in
\BGG_0$. Then $\SO P$ and $d_\BGG P$ are both injective hulls of
$L$. Hence they are isomorphic.

\medskip

\noindent \textbf{v)} Let $I \in \Inj(\BGG_0)$. Then $\A_\V I \in
\Inj \AB$. By Proposition \ref{injectivesec prop}  $\A_\V I$ is of
the form $\pi_\B J$ where $J \in \Inj(\AR)$.  Since the injectives
in $\mod(R)$ are precisely the direct sums of copies of $R$ we
obtain from  $(4)$ in Section \ref{roosterlabel} that  $J$ is of
the form $[A \to R^a \to R^b]$, for some  $A \in \mod(R)$. Thus
$$\A_\V I \cong [\ldots \to B^{-2} \to R^a \to R^b]$$
for some $B^i \in \B$. On the other hand let $P^\bullet \tto I$ be
any projective resolution of $I$ in $\BGG_0$. Then by definition
$\A_\V(I) = \V P^\bullet$. Since $\V(P_{w_0}) \cong R$ we obtain
$P^{-1} \cong P^a_{w_0}$ and $P^0 \cong P^b_{w_0}$.

 \medskip

\noindent \textbf{vi)} Let $M \in \BGG_0$ and $Q^{-1} \to Q^0 \to
M \to 0$ be a projective resolution. Then $d_\B \A_\V M \cong
\Ker(d_\B \A_\V  Q^{0} \to d_\B \A_\V
Q^{-1})$ and $\A_\V d_\BGG M \cong \Ker(\A_\V  d_\BGG Q^{0} \to \A_\V
d_\BGG Q^{-1})$. Thus it suffices to construct a functorial isomorphism
$\A_\V  d_\BGG P \isoto d_B \A_\V P$ for $P$ projective. By
\textbf{v)} we can pick a projective resolution $P^\bullet \tto
d_\BGG P$ of the form
$$
P^\bullet =  [\ldots \to P^{-2} \to P^a_{w_0}
\overset{\pa}{\longrightarrow} P^b_{w_0} ].
$$
Then by definition $\A_\V(d_\BGG P) = \V P^\bullet$. On the other
hand, since $\V d_\BGG P \cong \V P$ we get an exact sequence $R^a
\overset{\V \pa}{\longrightarrow} R^b \to \V P \to 0$. By
\eqref{Frobenius3} and \eqref{Frobenius4} we have $ d_R \V P \cong
[\Ker \V \pa \to R^a \to R^b]$ and therefore
$$d_\B \V P= \pi_\B d_R \taug_R \V P =  \pi_\B d_R \V P \cong \pi_\B[\Ker \V \pa \to R^a \to R^b].$$
The latter is evidently isomorphic to $ \V P^\bullet$. Thus  we
obtain $d_\B  \A_\V P = d_\B \V P \cong  \V P^\bullet.$ This
proves the assertion.

 \medskip

\noindent \textbf{vii)} follows from Corollary \ref{Frobenius
algebras section cor}.
\end{proof}
\begin{rem}\label{Soergel's theory rem} The existence of the Serre functor $\SO: \Proj(\BGG_0) \to \Inj(\BGG_0)$  is of course well-known and
follows from Eilenberg-Watts' representability theorem applied to
the functor $\Dd\Hom_{\BGG_0}(P, \- )$, for $P \in \Proj(\BGG_0)$.
The novelty here is how the Serre functor relates with AR duality
on $\mod(R)$.
\end{rem}
\subsection{Note about multiplicities in $\BGG$}\label{multisar}
We know from Proposition \ref{Section AA prop} that $$\Ker \pi_\B
= \{[X \to B \times Y \to Z] \in \AR\mid B \to Z \hbox{ is a }
\B \text{-}\hbox{cover}\}$$ is the kernel of the quotient functor
$\pi_\B: \AR \to \AB$.

By Corollary \ref{simple objects cor} a simple object $[\DTr X \to
Y \to X] \in \AR$ belongs to $\Ker \pi_\B$ if and only if $X
\notin \B$; for $X \in \B$ we have $[\DTr X \to Y \to X]_\C =
\L_X$. Therefore if one could solve the (difficult, but
\emph{only} depending on the AR theory of $\mod(R)$) problem of
describing all the simple subquotients of an object $V \in \AR$
one would in theory get a description of the simple subquotients
of $\pi_\B V$. This amounts to determine the multiplicities of
simple highest weight modules in $\A^{-1}_\V \pi_\B V$. Note that
$\pi_\B V$ has finite length, since $\BGG$ is a finite length
category, while $V$ in general has infinite length.

\medskip

An interesting case is that of a Verma module $M_x$. Let  $c_s =
{{1}\over{2}} (\alpha_s \ot 1 + 1 \ot \alpha_s) \in B_s$ and let
$K_s$ be the complex $R \to B_s$, $1 \mapsto c_s$, with $B_s$ in
degree $0$ and $R$ in degree $-1$. Let $\ul{x} = (s_m, \ldots,
s_1)$ be reduced. There is the\emph{ dual Rouquier complex} (see
\cite{EW})
 $$\Kx = K_{s_m} \ot_R \ldots \ot_R K_{s_1} \ot_R \mathbb{C}.$$
One can show that $\A_\V(M_x) \cong \Kx$. The complex $\Kx$ is
rather complicated but the truncated complex  $\taug_R \Kx$ equals
the somewhat simpler expression
$$\taug_R \Kx= [\Ker \phi \to \oplus^m_{i=1}  B_{s_m} \ot_R \ldots  \widehat{B_{s_i} \ot_R} \ldots  \ot_R B_{s_1} \ot_R \mathbb{C}
\overset{\phi}{\longrightarrow} \BSx].$$
 (See \cite{EW} Lemma 6.17 for an important property of this complex.)  Here $\widehat{B_{s_i} \ot_R}$ means removing this factor from the tensor product and $\phi$ is the differential in $\Kx$.  It would be interesting to know the simple subquotients of  $\taug_R \Kx$.

\subsection{{}}\label{Verma ext} Another interesting topic concerns $\Ext^i_{\BGG_0}(M_x, M_y)$. These groups
are not well understood at all, not even in the case $y=e$. In
theory the latter can be computed by means of the Rouquier
complex. Because, since $M_e$ is projective we have
\begin{equation}\label{first extis}
\Ext^i_{\BGG_0}(M_x, M_e) = \Hom_{\Db(\BGG_0)}(M_x,M_e[i]) \cong
\Hom_{\Db(\A_\B)}(\Kx, \k[i]).
\end{equation}

By Theorem \ref {injectivesec thm} $\pi_B$ has a right adjoint
denoted $\sigma$. It is left exact and maps injectives to
injectives. Thus we get an adjoint pair $\pi_\B: \Db(\AB)
\leftrightarrows \Db(\AR): R\sigma$. The duality functors $d_R$
and $d_\B$ from Section \ref{section frobenius algebras} naturally
extend to duality functors on $\T^b_R \cong \Db(\AR)$ and $\T^b_\B
\cong \Db(\AB)$, respectively. We denote these functors by the
same letters. Let $\iota: \T^b_\B \into \T^b_R$ be the inclusion
functor. Using the results of Section \ref{section frobenius
algebras} one can show that $R\sigma  \cong d_R \circ \iota \circ
d_\B$.  Since $\Kx = \pi_\B \circ \taug_R  \Kx$ we obtain that
\eqref{first extis} equals
$$
\Hom_{\Db(\A_\B)}(\pi_\B \taug_R \Kx, \k[i])  \cong
\Hom_{\Db(\A_R)}(\taug_R \Kx, R\sigma \k[i]) \cong$$
$$ \Hom_{\T^b_R}(\taug_R \Kx, d_R \circ \iota \circ d_\B \k[i]).
$$
Because of this it would be  interesting to describe the complex
$R\sigma \k = d_R \circ \iota \circ d_\B \k$. In the case when $\g
= \mathfrak{sl}_2$ then $\sigma$ is the identity functor so one
may hope that $R\sigma \k$ in general is somewhat simpler than the
dual Rouquier complex itself.

\end{document}